\documentclass[12pt, a4paper]{article}

\usepackage{subfigure}
\usepackage{graphicx}
\usepackage{natbib}
\usepackage{bm}
\usepackage{stmaryrd}
\usepackage{enumerate}
\usepackage{hyperref}
\usepackage{color}
\usepackage{lineno}
\usepackage{graphicx}
\usepackage{ae}
\usepackage{amsmath}
\usepackage{amssymb}
\usepackage{latexsym}
\usepackage{url}
\usepackage{epsfig}
\usepackage{mathrsfs}
\usepackage{amsfonts}
\usepackage{amsthm}
\usepackage{stmaryrd}
\usepackage{cite}
\usepackage{BOONDOX-cal}
\usepackage{indentfirst}
\usepackage{amssymb}
\usepackage{indentfirst}
\usepackage{algorithm}
\usepackage{algorithmic}
\numberwithin{equation}{section}
\allowdisplaybreaks

\def\qed{\hfill$\Box$\vspace{12pt}}

\usepackage{caption}
\usepackage{booktabs}
\usepackage{longtable}
\usepackage{multirow}
\usepackage{color}

\usepackage [T1]{fontenc}
\usepackage [utf8]{inputenc}
\usepackage{authblk}

\def\d{\mathrm{d}}

\newcommand{\E}{\mathbb{E}}
\newcommand{\R}{\mathbb{R}}

\newcommand{\lcx}{\preceq_{\mathrm{cx}}}

\renewcommand{\ge}{\geqslant}
\renewcommand{\le}{\leqslant}
\renewcommand{\geq}{\geqslant}
\renewcommand{\leq}{\leqslant}
\renewcommand{\epsilon}{\varepsilon}
\renewcommand{\cite}{\citet}

\usepackage[misc]{ifsym}

\newtheorem{theorem}{Theorem}[section]

\newtheorem{lemma}[theorem]{Lemma}
\newtheorem{proposition}[theorem]{Proposition}
\newtheorem{example}[theorem]{Example}

\newtheorem{remark}[theorem]{Remark}

\title{Exploratory mean–variance portfolio selection with Choquet regularizers}

\author[a]{Junyi Guo}
\author[a]{Xia Han}
\author[b]{Hao Wang\thanks{Corresponding author.

\ \ \text{E-mail addresses: jyguo@nankai.edu.cn (J. Guo); xiahan@nankai.edu.cn (X. Han);}

\ \ \text{hao.wang@mail.nankai.edu.cn (H. Wang)}}}
\affil [a] {School of Mathematical Sciences and LPMC, Nankai University, Tianjin, 300071, China}
\affil [b] {School of Mathematical Sciences, Nankai University, Tianjin, 300071, China}

\date{}

\begin{document}

\maketitle

\begin{abstract}
In this paper, we study a  continuous-time   exploratory mean-variance  (EMV) problem  under the framework of reinforcement learning (RL), and  the Choquet regularizers  are used to measure the level of exploration.  By applying the classical Bellman principle of optimality, the Hamilton–Jacobi–Bellman equation of the EMV  problem  is derived and   solved   explicitly via maximizing statically a mean–variance  constrained Choquet regularizer.  In particular,  the optimal distributions form a  location–scale family, whose shape depends on the choices of the Choquet regularizer. We further reformulate the continuous-time  Choquet-regularized EMV problem using a  variant of the  Choquet regularizer.    Several examples  are given  under   specific Choquet regularizers   that generate broadly used exploratory samplers such as  exponential, uniform and Gaussian. Finally, we  design a RL algorithm to simulate and compare  results  under the  two different forms of regularizers.

\bigskip


\bigskip

\end{abstract}

\section{INTRODUCTION}

 Reinforcement learning (RL) is an active subarea of machine learning. In RL, the agent can directly interact with the black box environment and get feedback. This kind of learning that focuses on the interaction process between the agent and the environment is called trial-and-error learning. By trial and error learning, we  skip the parameter estimation of the model and directly learn the optimal policy \citep{SB18}, which can overcome some  difficulties that traditional optimization theory may have in practice.    Many   RL algorithms are based on traditional deterministic optimization, and the optimal solution is usually a deterministic policy. But in some situations, it makes sense to solve for an optimal stochastic policy for exploration purposes. The stochastic policy is to change the determined action into a probability distribution through randomization.  Searching for the optimal stochastic policy has many advantages, such as robustness \citep{Z10} and better convergence \citep{GLGTL16} when the system dynamics are uncertain. 

Entropy  measures  the randomness of the actions an agent takes, and thus  can indicate the level of exploration in RL.  The idea of maximum entropy RL is to make the strategy more random in addition to maximizing the cumulative reward, so entropy together with a  temperature parameter   is added to the objective function as a regularization term; see e.g.,  \cite{NGJ17}.  Here, the temperature parameter  is a regularization coefficient used to control the importance of entropy;  the larger the parameter, the stronger the exploratory ability, which helps to accelerate the subsequent policy learning and reduces the possibility of the policy converging to a local optimum. \cite{HTAL17} generalized maximum entropy RL to continuous state and continuous action settings rather than tabular settings. \cite{WZZ20a} first established a continuous-time  RL framework with continuous state and action from the perspective of stochastic control and proved that the optimal exploration strategy for the linear–quadratic (LQ) control problem in the infinite time horizon is Gaussian. Further, \cite{WZ20} applied this RL framework for the first time to solve the continuous-time mean-variance (MV) problem, and we refer to \cite{Z21} for  more summaries. 
Motivated by   \cite{WZZ20a},  \cite{DDJ23} extended the exploratory stochastic control framework  to an incomplete
market, where the asset return correlates with a stochastic market state,  and learned  an equilibrium policy under  a mean-variance criterion. \cite{JSW22}  studied
the exploratory Kelly problem by considering both the amount
of investment in stock and the portion of wealth in stock as
the control for a general time-varying temperature
parameter. 

From the perspective of risk measures, \cite{HWZ23} first  introduced another kind of index that can measure the randomness of actions called Choquet regularization. They showed that  the optimal exploration distribution of LQ control problem with infinite time horizon is  no longer necessarily Gaussian as in \cite{WZZ20a}, but are dictated by the choice of Choquet
 regularizers.  As mentioned in  \cite{HWZ23},  Choquet regularizers have   a number  of theoretical and practical advantages to be  used  for RL. In particular, they satisfy several “good” properties such as quantile additivity, normalization, concavity, and consistency with convex  order (mean-preserving spreads) that facilitate analysis as regularizers.   Moreover,  the availability of a large class of Choquet regularizers makes it possible to compare and choose specific regularizers  to achieve certain objectives specific to each learning problem.   To the best of our knowledge, there is no literature  using other  regularizers rather than  entropy to quantify the information gain of exploring the environment   for practical problems.  Thus, it is natural to consider some practical exploratory stochastic control problems  using  the Choquet regularizers for regularization.

 This paper mainly studies the continuous-time  exploratory mean-variance (EMV)  problem as in  \cite{WZ20} in which we replace the differential entropy used for regularization with the Choquet regularizers. When looking for pre-committed optimal strategies as the goal, the MV model can be converted into a LQ model in finite time horizon  by \cite{ZL00}. The form of the LQ-specialized HJB equation suggests that the problem boils down to a static optimization where the given Choquet regularizer  is to be maximized over distributions with given mean and variance, which has been solved by \cite{LCLW20}.  Since the  EMV portfolio selection is  formulated in a finite time horizon,  we show that   the optimal distributions  form a location–scale family with a time-decaying variance whose shape  depends on the choice of Choquet regularizers.   This suggests that the level of exploration decreases as the time approaches the end of the planning horizon. We further give the optimal exploration strategies under several  specific Choquet regularizers, and  observe insights of  the perfect separation between exploitation and exploration in the mean and variance of the optimal distribution and the positive effect of a random environment on learning.
 
 Inspired by the form of entropy, we further reformulate the continuous-time Choquet-regularized RL problem  based on a variant of Choquet regularizers -- logarithmic Choquet regularizers.  Because of  the  monotonicity of the logarithmic function,  the problem can still be solved by  maximizing the  Choquet regularizer  over distributions with given mean and variance. However, since the   regularizers   affect the value function,   it is to be expected that the variance of the optimal distributions is  different.   Explicitly expressed costs of exploration for the two different forms of regularizers and close connections between the classical and the EMV problems are discussed.   It is interesting to see that  the costs of exploration for the two EMV problems are quite different.  To be specific,  with   the Choquet regularizers, the  exploration cost depends on the unknown model parameters and the  specific regularizers,   while   with logarithmic Choquet  regularizers, the derived exploration cost  only depends on  the exploration parameter and the time horizon, and it  is the same as the cost when using entropy as the regularizer in \cite{WZ20}. 
  
Finally, based on the  policy improvement and  convergence theorems,  we designed a RL algorithm to solve the   EMV problems  according to the continuous-time policy gradient method proposed by \cite{JZ22b} and then simulated it. By letting the  Choquet integral being   some concrete choices, we show that  our RL algorithm   based on Choquet regularizations and logarithmic Choquet regularizers perform on par with the one  in \cite{WZ20} where  the differential entropy is applied and  Gaussian is always the optimal exploration distribution.

The rest of this paper is organized as follows. Section \ref{sec:2} introduces the MV problem under the Choquet regularizations. Section \ref{sec:3} solves the   continuous-time  EMV problem  and gives several examples. Section \ref{sec:4} discusses  the corresponding results under  the variant of Choquet regularizations. Section \ref{sec:5} introduces the RL algorithm, and the simulation results of the algorithm are summarized in Section \ref{sec:6}. Section \ref{sec:7} concludes the paper.
\section{FORMULATION OF PROBLEM}\label{sec:2}

\subsection{Choquet regularizers}
We assume that $(\Omega, \mathcal{F}, \mathbb{P})$ is an atomless probability space. With a slight abuse of notation,
let 
$\mathcal M$  denote both the set of (probability) distribution functions of real random variables and the set of Borel probability measures on $\mathbb{R}$, with
the obvious identity $\Pi(x)\equiv  \Pi((-\infty, x])$   for $x \in \mathbb{R}$
and
$\Pi \in \mathcal{M}$. We denote by  $\mathcal M^p\subset \mathcal M$,  $p\in[1,\infty)$,  the set of distribution functions or probability measures with finite $p$-th moment.   For a random variable $X$ and a distribution $\Pi$,  we write $X \sim \Pi$ if the distribution of  $X$ is $\Pi$ under $\mathbb{P}$, and $X \stackrel{\rm d}{=} Y$ if two random variables $X$ and $Y$ have the same distribution.
 We denote by $\mu$ and $\sigma^2$  the mean and  variance functionals on  $\mathcal{M}^2$, respectively; that is, $\mu(\Pi)$ is the mean of $\Pi$ and $\sigma^2(\Pi)$  the variance of $\Pi$ for $\Pi \in\mathcal M^2$.  We denote by $\mathcal{M}^2(m,s^2)$    the set of $\Pi \in\mathcal M^2$  satisfying  $\mu(\Pi)=m\in\R$ and $\sigma^2(\Pi)=s^2>0$.

In \cite{HWZ23},  the  {\it Choquet regularizer} is defined to   measure and manage the level of exploration for RL  based on a subclass of signed Choquet integrals \citep{WWW20b}. Given a concave function $h:[0,1]\to \R$  of bounded variation with $h(0)=h(1)=0$ and   $\Pi\in \mathcal M$,
  the Choquet regularizer $\Phi_h$  on $\mathcal M$ is defined as 
\begin{equation*}
\Phi_h(\Pi)\equiv \int  h\circ \Pi([x,\infty))\d x:=\int_{-\infty}^{0}\left[h\circ \Pi ( [x,\infty) )-h(1)\right]\d x+\int_0^\infty h\circ \Pi([x,\infty))\d x. 
\end{equation*}
Note that  the concavity of $h$ is equivalent to several other properties, and in particular,
to that $\Phi_h$ is a concave mapping which means that 
$$\Phi_h(\lambda \Pi_1 + (1-\lambda) \Pi_2 ) \ge \lambda \Phi_h( \Pi_1)  + (1-\lambda) \Phi_h(\Pi_2 ), ~~\mbox{for all~}\Pi_1,\Pi_2\in \mathcal M\mbox{~and~}\lambda\in [0,1],$$
and consistency with convex order means
$$\Phi_h( \Pi_1  ) \le  \Phi_h(\Pi_2 ), ~~\mbox{for all~}\Pi_1,\Pi_2\in \mathcal M \mbox{~with~} \Pi_1\lcx \Pi_2.\footnote{ $\Pi_1$ is smaller than $\Pi_2$  in \emph{convex order}, denoted by  $ \Pi_1\lcx \Pi_2$, if   $\mathbb{E}[f( \Pi_1)] \leq \mathbb{E}[f( \Pi_2)]$ for all convex functions $f$,  provided that the two expectations exist. It is immediate that $\Pi_1 \lcx \Pi_2$ implies $\mathbb{E}[\Pi_1]\leq\mathbb{E}[\Pi_2]$.}$$
If $\Pi_1\lcx \Pi_2$, then  $\Pi_2$ is also called a  {\it mean-preserving spread} of $\Pi_1$, which intuitively means that $\Pi_2$ is more spread-out (and hence ``more random")  than $\Pi_1$. The set of $h:[0,1]\to \R$ is denoted by   $\mathcal H$.

We remark that the above  properties  indeed suggest  that $\Phi_h(\Pi)$ serves as a measure of randomness for  $\Pi$, since both a mixture and a mean-preserving spread introduce extra randomness. 
 On the other hand,  $h(0)=h(1)=0$  is equivalent to $\Phi_h(\delta_c)=0$, $\forall c\in \R$, where $\delta_c$ is the Dirac measure at $c$. That is, degenerate distributions do not have any randomness measured by $\Phi_h$. 
  Choquet regularizers include, for instance,  range,  mean-median deviation, the Gini deviation, and  inter-ES differences; see Section 2.6 of \cite{WWW20b}.  

By Lemma 2.2 of \cite{HWZ23},  $\Phi_h$  is  well defined, non-negative,  and
 location invariant and scale homogeneous for $h\in \mathcal H$.\footnote{We  call $\Phi_h$ to be location invariant and scale homogeneous if
  $\Phi_h(\Pi')=\lambda \Phi_h (\Pi)$
 where $\Pi' $ is the distribution of $\lambda X+c$ for $\lambda >0$, $c\in \R$ and $X\sim \Pi$.}  The properties imply that  any distribution for exploration can be measured in non-negative values. Moreover, the measurement of randomness does not depend on the location and  is linear in its scale, which  make $\Phi_h$ a meaningful  regularizer that measures  the level of  randomness, or the level of exploration in the RL context.

\medskip

For a distribution $\Pi\in\mathcal{M}$,  let its   left-quantile for $p\in(0,1]$ be defined as $$Q_\Pi(p)=\inf \left\{x\in \R: \Pi(x) \ge p\right\} .$$
  Next,  we give  a lemma which we will  rely on  when considering the  EMV problem formulated by \cite{WZ20}. Let  $h'$ be  the right-derivative of $h$ and $\Vert h'\Vert_2=\left(\int_0^1(h'(p))^2dp\right)^{1/2}$.
\begin{lemma}[Theorem 3.1 of \cite{LCLW20}]\label{lem:liu}
If $h$ is continuous and not constantly zero, then
a maximizer $\Pi^*$ to the optimization problem \begin{align}\label{eq:opt}
\max_{\Pi \in \mathcal M^2}\Phi_h (\Pi) \mbox{~~~~~subject to~} \mu (\Pi) =m ~\mbox{and}~ \sigma^2 (\Pi) = s^2
\end{align}  has the following quantile function
\begin{equation}\label{eq:lemma3}
Q_{\Pi^*}(p) =  m + s\frac{ h'(1-p) }{ ||h'||_2}, ~~ \mbox{~a.e. }p\in (0,1),
\end{equation}
and the  maximum value of \eqref{eq:opt} is $\Phi_h(\Pi^*)= s||h'||_2$. 
\end{lemma}
By Lemma \ref{lem:liu}, \cite{HWZ23}  presented many examples linking specific exploratory distributions with the corresponding Choquet regularizers and  generated some common exploration measures including $\epsilon$-greedy, three-point, exponential, uniform and Gaussian; see their Examples 4.3--4.6 and Sections 4.3--4.5. 
\begin{remark}\label{rem:general} The result in Lemma \ref{lem:liu} can be extended to a more general case involving higher moments.  For $a>1$, Theorem 5 in \cite{PWW20} showed that if the uncertain set is given by
 $$\mathcal{M}^{a}(m, v)=\left\{\Pi \in \mathcal{M}^a: \mu(\Pi)=m~\mbox{and}~ \mathbb{E}\left[|\Pi-m|^{a}\right] \leq v^{a}\right\},$$
 the optimization problem
 $\max_{\Pi \in \mathcal M^{a}}\Phi_h (\Pi) $, for $p\in(0,1)$, can be solved by 
\begin{equation*}
Q_\Pi(p) =  m + v \frac{\left|h^{\prime}(1-p)-c_{h, b}\right|^{b}}{h^{\prime}(1-p)-c_{h, b}}[h]_{b}^{1-b}, ~\mbox{ if }~h^{\prime}(1-p)-c_{h, b} \neq 0, ~ \mbox{ and }Q_\Pi(p) =  m~  \mbox{otherwise}.
\end{equation*}   Here,  $ b \in[1, \infty]$ is the H\"older conjugate of $a$, namely $b=(1-1 / a)^{-1}$, or equivalently, $1/a+1/b=1$,
$$
c_{h, b}=\underset{x \in \mathbb{R}}{\arg \min }\left\|h^{\prime}-x\right\|_{b} \quad \text { and } \quad[h]_{b}=\min _{x \in \mathbb{R}}\left\|h^{\prime}-x\right\|_{b}=\left\|h^{\prime}-c_{h, b}\right\|_{b},
$$ with $$
\left\|h^{\prime}-x\right\|_{b}=\left(\int_{0}^{1}\left|h^{\prime}(p)-x\right|^{b} \mathrm{~d} p\right)^{1 / b}, ~b<\infty \text {~ and }~\left\|h^{\prime}-x\right\|_{\infty}=\max _{p \in[0,1]}\left|h^{\prime}(p)-x\right|,~ x \in \mathbb{R}.
$$
  \end{remark}

\subsection{Continuous-time  EMV problem}\label{sec:2.2}

The classical MV problem has been well studied in the literature; see e.g., \cite{M52}, \cite{LN00} and \cite{LZL02}.  We  first briefly introduce the classical MV problem in continuous time.
 
   Let $T$ be a fixed investment planning horizon and $\{W_t,0\leqslant t \leqslant T\}$ be a standard  Brownian motion defined on a given filtered probability space $(\Omega,\mathscr{F},\{\mathscr{F}_t\}_{0\leqslant t\leqslant T},\mathbb{P})$ that statisfies usual conditions. 
Assume that a financial market  consists of a riskless asset and only one risky asset,  where the riskless asset has a constant interest rate $r>0$ and the risky asset has a price process governed by
\begin{equation}\label{gb}
	\d S_t=S_t(\mu \d t+\sigma \d W_t),\ \ ~~~ 0\leqslant t
\leqslant T,
\end{equation}
with $S_0=s_0>0$ where  $\mu \in \mathbb{R},\sigma >0$ is  the mean and volatility parameters, respectively. The Sharpe ratio of the risky asset is defined by $\rho={(\mu-r)}/{\sigma}$. Let $u=\{u_t,0\leqslant t \leqslant T\}$ denote the discounted amount invested in the risky asset at time  $t$, and the rest of the wealth is invested in the risk-free asset.  By \eqref{gb},  the discounted wealth process $\{X^u_t,0\leqslant t \leqslant T\}$  for a strategy $u_t$ is then given as
\begin{align}\label{dv}
	\d X_t^u&=\sigma u_t(\rho \d t+\d W_t),\ \ ~~~ 0\leqslant t
\leqslant T,
\end{align}
with $X_0^u=x_0\in \mathbb{R}$.  Under the continuous-time MV  setting, we aim  to solve the following constrained optimization problem
\begin{align}\label{cmv}
	\begin{split}
	&\min\limits_u \mathrm{Var}[X_T^u]~~~~~\text{subject to } \mathrm E[X_T^u]=z,
	\end{split}
\end{align}
where $\{X_t^u,0\leqslant t \leqslant T\}$ satisfies the dynamics \eqref{dv} under the investment strategy $u$, and $z \in \mathbb R$ is an investment target determined  at $t=0$ as the desired mean payoff at the end of the investment horizon $[0,T]$.

By applying a Lagrange multiplier $w$, we can transform \eqref{cmv} into an unconstrained problem 
\begin{align}\label{object}
\min\limits_u\mathrm E[(X^u_T)^2]-z^2-2w(\mathrm E[X_T^u]-z)=\min\limits_u\mathrm E[(X^u_T-w)^2]-(w-z)^2.
\end{align}	The problem in  \eqref{object} was well studied by  \cite{LN00},  and it  can be solved analytically, whose solution $u^*$ depends on $w$. Then the original constraint $\mathrm E[X_T^{u^*}]=z	$ determines the value of $w$.

 Employing the method in \cite{WZZ20a} and \cite{WZ20}, we  give the ``exploratory" version of the state dynamic \eqref{dv}  motivated by repetitive learning in RL.  In this formulation, the control process is now randomized, leading to a distributional or exploratory control process  denoted by $\Pi=\{\Pi_t,0\leqslant t \leqslant T\}$. Here,  $\Pi_{t}\in \mathscr{M}(U)$ is the probability distribution function for control at time $t$,   with  $\mathscr{M}(U)$  being the set of   distribution functions on $U$.  For  such  a given distributional control  $\Pi\in\mathscr{M}(U)$,  the exploratory version of the state dynamics in \eqref{dv} is changed to
 \begin{align}\label{emv}
	\d X_t^{\Pi}=\widetilde{b}(\Pi_t)\d t+\widetilde{\sigma}(\Pi_t)\d W_t,\ \ 0<t\leqslant T,
\end{align}
with $X_0^{\Pi}=x_0$, where
\begin{align}\label{b}
	\widetilde{b}(\Pi):=\int_{\mathbb R}\rho \sigma u \d\Pi (u)~~~~\text{and}~~~~
	\widetilde{\sigma}(\Pi):=\sqrt{\int_{\mathbb R}\sigma^2 u^2 \d \Pi(u)}.
\end{align}

Denote the mean and variance processes associated with the control process $\Pi $ by $\mu_t$ and $\sigma^2_t$ for $0\leqslant t\leqslant T$: 
\begin{align}\label{meanvar1}
	\begin{split}
		\mu_t: =\int_{\mathbb R}u \d\Pi_t(u), 
		~~~\sigma^2_t:&=\int_{\mathbb R}u^2\d \Pi_t(u)-\mu_t^2. 
	\end{split}
\end{align}
Then it follows from \eqref{emv}--\eqref{meanvar1} that \begin{align}\label{emv2}
	\begin{split}
		\d X_t^{\Pi }=\rho\sigma\mu_t\d t+\sigma\sqrt{\mu_t^2+\sigma_t^2}\d W_t,
	\end{split}
\end{align}
	with $X_0^{\Pi }=x_0$. We refer to  \cite[pp. 6--8]{WZZ20a} for more detailed explanation of where this exploratory formulation comes from.

Next, we use a Choquet regularizer  $\Phi_h$ to measure the level of exploration, and  the aim of the exploratory control is to achieve a continuous-time EMV problem  under the framework of RL. For any fixed $w \in \mathbb R$,    we get the Choquet-regularized EMV problem by adding an exploration weight $\lambda>0$, which reflects the strength of the exploration desire: \begin{align*}
	\min\limits_{\Pi \in\mathscr{A}(0,x_0)} \mathrm{E}\left[(X_T^{\Pi }-w)^2-\lambda\int_0^T\Phi_h(\Pi_t)\d t\right]-(w-z)^2,
\end{align*}
where    $\mathscr{A}(t,x)$   is  the set of all admissible controls $\Pi $ for  $(t,x)\in [0,T)\times \mathbb R$.   A  control process  $\Pi\in\mathscr{A}(t,x)$ is said to be admissible  if (i) for $t\leqslant s\leqslant T$, $\Pi_s \in \mathscr{M}({\mathbb R})$ a.s.; (ii) for $A \in \mathscr{B}(\mathbb R),\ \{\int_A\Pi_s(u)\d u,t \leqslant s \leqslant T\}$ is $\mathscr{F}_s$-progressively measurable;
	(iii) $\mathrm E[\int_t^T(\mu_s^2+\sigma_s^2)\d s]<\infty$; and
(iv) $\mathrm{E}[(X_T^{\Pi }-w)^2-\lambda\int_t^T\Phi_h(\Pi_s)\d s\large|X_t^{\Pi }=x]<\infty$.

The value function is then defined as
\begin{align}\label{vf}
	V(t,x;w):=\inf\limits_{\Pi \in\mathscr{A}(t,x)}\mathrm{E}\left[(X_T^{\Pi }-w)^2-\lambda\int_t^T\Phi_h(\Pi _s)ds|X_t^{\Pi }=x\right]-(w-z)^2,
\end{align}
and the value function under feedback control $\Pi $ is
\begin{align}\label{vfp}
	V^{\Pi }(t,x;w):=\mathrm{E}\left[(X_T^{\Pi }-w)^2-\lambda\int_t^T\Phi_h(\Pi _s)\d s|X_t^{\Pi }=x\right]-(w-z)^2.
\end{align}
\section{SOLVING EMV PROBLEM}\label{sec:3}

In this section, we aim to to solve the Choquet-regularized EMV   problem. 
 Firstly, we have following result based on Lemma \ref{lem:liu}.
 \begin{proposition}\label{prop:EMV_prob}
Let a continuous $h\in \mathcal H$ be given.
 For any  $\Pi=\{\Pi_t\}_{t\ge 0}\in \mathcal A(t,x)$ with mean process $\{\mu_t\}_{t\ge 0}$ and variance process $\{\sigma_t^2\}_{t\ge 0}$, there exists  $\Pi^*=\{\Pi^*_t\}_{t\ge 0}\in \mathcal A(t,x)$ given by
 \begin{equation*}
Q_{\Pi^*_t}(p) = \mu_t + \sigma_t\frac{ h'(1-p) }{ ||h'||_2}, ~~ ~~\mbox{~a.e.}~p\in (0,1),\;\;t\geq0,
\end{equation*}
 which has the same mean  and variance processes satisfying
 $V^{\Pi^* }(t,x;w)\leqslant V^{\Pi }(t,x;w)$.

 \end{proposition}
\begin{proof}
 By \eqref{emv}, it is clear  that the term $\mathrm{E}\left[(X_T^{\Pi }-w)^2|X_t^{\Pi }=x\right]$ in \eqref{vfp}   only depends on  the mean process $\{\mu_t\}_{t\ge 0}$ and the variance process $\{\sigma_t^2\}_{t\ge 0}$ of $\{\Pi_t\}_{t\ge 0}$. Thus,
for any fixed $t\geq 0$, choose $\Pi_t^*$ with mean $\mu_t$ and variance $\sigma_t^2$
that maximizes $\Phi_h(\Pi)$. Together with  Lemma \ref{lem:liu}, we get the desired result.   \end{proof}
Proposition \ref{prop:EMV_prob} indicates that the control  problem in  \eqref{vf}  
 is  maximized within  a location--scale family of distributions,\footnote{Recall that given a distribution $\Pi$ the  {\it location-scale family} of   $\Pi$  is the set of all distributions $\Pi_{a,b}$ parameterized by  $a\in\R$ and $b>0$ such that  $\Pi_{a,b}(x)=\Pi((x-a)/b)$  for all $ x \in \mathbb{R}$} which is determined only by $h$.
\begin{remark} We know from  Remark \ref{rem:general} that if  both the reward term and the dynamic process only depend on the mean process $\mu_t$ and the $
 a$-th moment  process $\sigma^{a}_t$  of  $\Pi_t$ for $ t\geq 0$, then we have  $V^{\Pi^* }(t,x;w)\leqslant V^{\Pi }(t,x;w)$ 
with  $\Pi^{*}_t$   satisfying   $$
Q_{\Pi_t^*}(p) =  \mu_t + \sigma_t \frac{\left|h^{\prime}(1-p)-c_{h, b}\right|^{b}}{h^{\prime}(1-p)-c_{h, b}}[h]_{b}^{1-b}, ~\mbox{ if }~h^{\prime}(1-p)-c_{h, b} \neq 0, ~ \mbox{ and }Q_{\Pi_t^*}(p) =  \mu_t,  \mbox{otherwise}.
$$    \end{remark}
Using the Bellman's dynamic principle, we get
\begin{align}\label{dp}
	V(t,x;w)=\inf\limits_{\Pi \in\mathscr{A}(t,x)}\mathrm{E}\left[-\lambda\int_t^s\Phi_h(\Pi_v)dv+V(s,X_s^{\Pi };w)|X_t^{\Pi }=x\right].
\end{align}
Then we can deduce  from \eqref{dp}
 that $V$ satisfies the HJB equation \begin{align}\label{hjb1}
	V_t(t,x;w)+\min\limits_{\Pi\in \mathscr{M}(\mathbb R)}	\left[\dfrac{1}{2}\widetilde{\sigma}^2(\Pi)V_{xx}(t,x;w)+\widetilde{b}(\Pi)V_x(t,x;w)-\lambda\Phi_h(\Pi)\right]=0.
\end{align}
By \eqref{b}, the HJB equation in \eqref{hjb1} is equivalent to 
\begin{align}\label{hjb}
	V_t(t,x;w)+\min\limits_{\Pi\in \mathscr{M}(\mathbb R)}\left[\dfrac{\sigma^2}{2}\left(\mu(\Pi)^2+\sigma(\Pi)^2\right)V_{xx}(t,x;w)+\rho \sigma \mu(\Pi)V_x(t,x;w)-\lambda\Phi_h(\Pi)\right]=0,
\end{align}
with terminal condition $V(T,x;w)=(x-w)^2-(w-z)^2$. Here, we assume that $\Pi$ has finite second-order moment, and $\mu(\Pi)$ and $\sigma(\Pi)^2$ are the mean and variance of $\Pi$, respectively.

We now  pay attention to the minimization in \eqref{hjb}.  Let
\begin{align*}
	\varphi(t,x,\Pi)=\dfrac{\sigma^2}{2}\left(\mu(\Pi)^2+\sigma(\Pi)^2\right)V_{xx}(t,x;w)+\rho \sigma \mu(\Pi)V_x(t,x;w)-\lambda\Phi_h(\Pi).
\end{align*}
Note that $\varphi(t,x,\Pi)$ only depends on $\Pi$ by $\mu(\Pi)$ and $\sigma(\Pi)^2$ except $\Phi_h(\Pi)$, we get
\begin{align*}
	\min\limits_{\Pi\in \mathscr{M}(\mathbb R)}\varphi(t,x,\Pi)=\min\limits_{m\in \mathbb R,s>0}\min\limits_{\substack{\Pi\in \mathscr{M}(R)\\ \mu(\Pi)=m,\sigma(\Pi)^2=s^2}}\varphi(t,x,\Pi),
\end{align*}
and the inner minimization problem is equivalent to
\begin{align}\label{mcr}
	\max\limits_{\Pi\in \mathscr{M}(R)}\Phi_h(\Pi)\ \ \text{subject to }\mu(\Pi)=m,\ \sigma(\Pi)^2=s^2.
\end{align}
By Lemma \ref{lem:liu}, the maximizer $\Pi^*$ of \eqref{mcr} whose quantile function is $Q_{\Pi^*}(p)$ satisfies
\begin{align}\label{qf}
	Q_{\Pi^*}(p)=m+s\dfrac{h'(1-p)}{\Vert h'\Vert_2},
\end{align}
and
	$\Phi_h(\Pi^*)=s\Vert h'\Vert_2.$ 
Then  the HJB equation in \eqref{hjb} is converted to
\begin{align}\label{hjb2}
	V_t(t,x;w)+\min\limits_{m\in\mathbb R,s>0}\left[\dfrac{\sigma^2}{2}(m^2+s^2)V_{xx}(t,x;w)+\rho\sigma mV_x(t,x;w)-\lambda s\Vert h'\Vert_2\right]=0.
\end{align} 
By the first-order conditions, we get the minimizer of \eqref{hjb2}
\begin{align}\label{ms}
	m^*=-\dfrac{\rho}{\sigma}\dfrac{V_x}{V_{xx}},~~~~\ \text{and } ~~~s^*=\dfrac{\lambda \Vert h' \Vert_2}{\sigma^2v_{xx}}.
\end{align}
Bringing $m^*$ and $s^*$ back into \eqref{hjb2}, we can rewrite \eqref{hjb2} as
\begin{align}\label{hjb3}
	V_t-\dfrac{\rho^2}{2}\dfrac{V_x^2}{V_{xx}}-\dfrac{\lambda^2}{2\sigma^2}\dfrac{\Vert h'\Vert_2^2}{V_{xx}}=0.
\end{align}
By the terminal condition $V(T,x;w)=(x-w)^2-(w-z)^2$, a smooth solution to  \eqref{hjb3} is given by
\begin{align}\label{solution}
	V(t,x;w)=(x-w)^2e^{-\rho^2(T-t)}-\dfrac{\lambda^2\Vert h'\Vert_2^2}{4\rho^2\sigma^2}(e^{\rho^2(T-t)}-1)-(w-z)^2.
\end{align}
Then we can deduce from \eqref{qf}, \eqref{ms} and \eqref{solution} that 
$$
	m^*=-\dfrac{\rho}{\sigma}(x-w)\label{m}, ~~~~\text{and}~~~~
	s^*=\dfrac{\lambda\Vert h'\Vert_2}{2\sigma^2}e^{\rho^2(T-t)}\label{s},$$
and the dynamic \eqref{emv2} under $\Pi^*$ becomes 
$$\d X_t^*=-\rho^2(X_t^*-w)\d t+\sqrt{\rho^2(X_t^*-w)^2+\dfrac{\lambda^2\Vert h'\Vert^2_2}{4\sigma^2}e^{2\rho^2(T-t)}}\d W_t$$
with $X_0^*=x_0$.

Finally, we try to calculate $w$. By $\mathrm E[\max\limits_{t\in[0,T]}(X_t^*)^2]<\infty$ and using Fubini theorem,  we get
\begin{align*}
	\mathrm E[X_t^*]=x_0+\mathrm E\left[\int_0^t-\rho^2(X_s^*-w)ds\right]=x_0+\int_0^t-\rho^2(\mathrm E[X_s^*]-w)\d s.
\end{align*}
Hence, $\mathrm E[X_t^*]=(x_0-w)^2e^{-\rho^2t}+w$. It follows from $\mathrm E[X_T^*]=z$ that $$w=\frac{ze^{\rho^2T}-x_0}{e^{\rho^2T}-1}.$$

We  summarize the above results in the following theorem.
\begin{theorem}\label{nolog}
	The value function of  Choquet-regularized EMV problem in \eqref{vf} is given by
	\begin{align}\label{valuefunction}
		V(t,x;w)=(x-w)^2e^{-\rho^2(T-t)}-\dfrac{\lambda^2\Vert h'\Vert_2^2}{4\rho^2\sigma^2}(e^{\rho^2(T-t)}-1)-(w-z)^2,
	\end{align}
	and the corresponding optimal control process is $\Pi^*$, whose quantile function is 
	\begin{align}\label{quantilefunction}
	Q_{\Pi^*}(p)=-\dfrac{\rho}{\sigma}(x-w)+\dfrac{\lambda h'(1-p)}{2\sigma^2}e^{\rho^2(T-t)},	
	\end{align} with the mean and variance of $\Pi^*$
   \begin{equation}\label{variance}
		\mu(\Pi^*)=-\dfrac{\rho}{\sigma}(x-w),~~~~\text{and}~~~~
	\sigma(\Pi^*)^2=\dfrac{\lambda^2\Vert h'\Vert_2^2}{4\sigma^4}e^{2\rho^2(T-t)}.
   \end{equation}
	The optimal wealth process under $\Pi^*$ is the unique solution of the SDE
	\begin{align*}
		\begin{split}
		\d X_t^*=-\rho^2(X_t^*-w)\d t+\sqrt{\rho^2(X_t^*-w)^2+\dfrac{\lambda^2\Vert h'\Vert^2_2}{4\sigma^2}e^{2\rho^2(T-t)}}\d W_t\
	    \end{split}
	\end{align*}
with 	    $x_0^*=x_0$. 	Finally, the Lagrange multiplier $w$ is  given by $$w=\frac{ze^{\rho^2T}-x_0}{e^{\rho^2T}-1}.$$
\end{theorem}
\begin{proof}
	Along with the similar lines of the verification  theorem  in \cite {WZZ20a} (see their Theorem 4), we can verify that for any $w\in\mathbb R$,  \eqref{valuefunction} is  indeed the value function and the optimal control $\Pi^*$ is admissible.\qed
\end{proof}

There are several observations to note in this result.  We    can see from  \eqref{quantilefunction}  that
for any  Choquet regularizer,  the optimal exploratory distribution    is  uniquely determined   by $h'$.  Different $h$  corresponds to a different Choquet regularizer; hence $h$ will certainly   affect the way and the level of exploration.  Also,   since $h'(x)$ is the ``probability weight" put on $x$ when calculating the (nonlinear) Choquet expectation; see e.g.,   \cite{GS89} and \cite{Q82},  the more weight put on the level of exploration, the more spreaded out the exploration becomes around the current position.  In addition, we point out that if we fix the value of $\Vert h'\Vert^2_2$ for  different Choquet regularizers by multiplying or dividing by a constant, the mean and variance of the different optimal distributions are equal. 

Moreover,  the optimal control processes under $\Phi_h$ has the same expectation  as the one in \cite{WZ20} when the differential  entropy is used as a regularizer, which is also identical to the optimal control of the classical, non-exploratory MV problem, and the expectation is independent of $\lambda$ and $h$. Meanwhile,  the variance of optimal control process is independent of state $x$ but decreases over time, which is different from \cite{HWZ23} where an infinite horizon counterpart is studied. This is intuitive because by exploration,  one can  get more  information over time,  and then the demand and aspiration of exploration decreases.  In a sense, the expectation represents exploitation which means making the best decision based on existing information, and the variance represents exploration.  As a result, the
 observations above  show a perfect separation between exploitation
 and exploration.

In the following example,   we show optimal exploration samplers under the EMV framework for some  concrete choices of $h$ studied in \cite{HWZ23}. Theorem \ref{nolog} yields that the mean of the  optimal distribution is independent of $h$,  so we will specify only  its quantile function and variance  for each  $h$ discussed below. 
\begin{example}\label{exm:3.4}
(i) Let $h(p)=-p\log(p)$. Then we have 
    $$\Phi_h (\Pi)=\int_0^\infty  \Pi([x,\infty)) \log(\Pi([x,\infty)))\d x,$$   which is  the cumulative residual entropy defined in \cite{HC20} and \cite{RCVW04};  see  Example 4.5  of \cite{HWZ23}.   The  optimal  policy is   a shifted-exponential  distribution   given as
$$\Pi^*(u; t,x)=1-\exp\left\{-\frac {2\sigma^2}{\lambda e^{\rho^2(T-t)}}\left(u+\frac{\rho}{\sigma}(x-w)\right) -1\right\}.$$ Since  $\|h'\|^2_2=1$, the variance of $\Pi^*$ is given by $$(\sigma^*(x))^2=\dfrac{\lambda^2}{4\sigma^4}e^{2\rho^2(T-t)}.$$

(ii) Let $h(p)=\int_0^p z(1-s)\d s$, where $z$ is the standard normal quantile function. We have  $\Phi_h (\Pi) =\int_0^1 Q_\Pi(p) z (p) \d p$; see Example 4.6 of \cite{HWZ23}.    The  optimal  policy  is a normal   distribution  given by
$$
{\Pi}^{*}(\cdot ; t,x)=  {\mathrm N}\left(-\dfrac{\rho}{\sigma}(x-w), \dfrac{\lambda^2}{4\sigma^4}e^{2\rho^2(T-t)}\right),
$$  owing to the fact that $\|h'\|^2_2=1$.

(iii)
Let $h(p)=p-p^2$. Then  $
  \Phi_h(\Pi) =  \E[|X_1-X_2|]/2$, which is the  Gini mean difference; see Section 4.5 of \cite{HWZ23}.  The  optimal  policy  $\Pi^*(\cdot;x)$ is  a uniform  distribution   given as
  $$\mathrm{U}\left[-\dfrac{\rho}{\sigma}(x-w)-\dfrac{\lambda}{2\sigma^2}e^{\rho^2(T-t)},-\dfrac{\rho}{\sigma}(x-w)+\dfrac{\lambda}{2\sigma^2}e^{\rho^2(T-t)}\right].$$
Since $\|h'\|^2_2=1/3$, the variance of $\Pi^*$ is given by $(\sigma^*(x))^2={\lambda^2e^{2\rho^2(T-t)}}/{12\sigma^4}.$

\end{example}

\section{An alternative form of  Choquet regularizers}\label{sec:4}

As mentioned in Introduction, for an absolutely continuous $\Pi$,  Shannon's differential entropy, defined  as \begin{equation*}\label{eq:DE}{\rm DE}(\Pi):=-\int_{\R}\Pi'(x)\log(\Pi'(x))\d x\end{equation*} 
is commonly used for exploration--exploitation balance in  RL; see  \cite{WZ20},  \cite{JSW22} and \cite{DDJ23}.  It admits a different quantile representation (see \cite{SS12}) 
\begin{equation*}
	{\rm DE}(\Pi)= \int_0^1 \log (Q'_\Pi(p)) \d p.
\end{equation*}
It is clear that DE is location invariant, but not scale homogeneous. It is not quantile additive either. Therefore,  DE is {\it not} a Choquet regularizer.

 Inspired by the logarithmic form of DE, we  consider another   EMV problem:
\begin{align}\label{dp1}
	\widehat V(t,x;w):=\inf\limits_{\Pi \in\mathscr{A}(t,x)}\mathrm{E}\left[(X_T^{\Pi }-w)^2-\lambda\int_t^T\log\Phi_h(\Pi_s)ds|X_t^{\Pi }=x\right]-(w-z)^2,
\end{align}
where we apply the  logarithmic form of $\Phi_h$  as  the regularizer   to   measure and manage the level of exploration. According to the monotonicity and concavity of logarithmic function, we can easily verify that $\log\Phi_h$ is  still  a concave mapping:
$$\log\Phi_h(\lambda \Pi_1 + (1-\lambda) \Pi_2 ) \ge \log(\lambda \Phi_h( \Pi_1)  + (1-\lambda) \Phi_h(\Pi_2 )) \ge \lambda\log\Phi_h(\Pi_1)+(1-\lambda)\log\Phi_h(\Pi_2)$$ $\mbox{for all~}\Pi_1,\Pi_2\in \mathcal M\mbox{~and~}\lambda\in [0,1],$
and consistent with convex order:  $$\log\Phi_h( \Pi_1  ) \le  \log\Phi_h(\Pi_2 ), ~~~~\mbox{for all~}\Pi_1,\Pi_2\in \mathcal M \mbox{~with~} \Pi_1\lcx \Pi_2.$$
Comparing to the properties of $\Phi_h$, $\log\Phi_h$ is not necessarily non-negative as $\Phi_h$.    However,  the non-negativity does not inherently affect the exploration.  Further, $\Phi(\Pi)$ is zero when $\Pi$ is Dirac measure, we  then have $\log\Phi(\delta_c)=-\infty$ for all $c\in \mathbb R$.  The location invariance for $\log\Phi_h$ is obvious. For scale homogeneity, $\log\Phi_h$ is no longer linear in its scale, but  we have  $\log\Phi_h(\Pi')=\log\Phi_h(\Pi)+\log\lambda$ for any $\lambda>0$  where $\Pi' $ is the distribution of $\lambda X$ for $\lambda >0$ and $X\sim \Pi$.    It is interesting to see that the level of  randomness is  captured by the term  of $\log\lambda$.  Based on the observations above, we find that $\log\Phi_h$ has many similarities with DE in capturing the randomness.

We remark that  maximizing $\Phi_h$  over $\mathcal{M}^2(m,s^2)$ is equivalent to maximizing  $\log\Phi_h$ over $\mathcal{M}^2(m,s^2)$.  In the following theorem, we give the optimal result of \eqref{dp1} directly.  Since the procedure is similar to Section \ref{sec:3}, we omit the details here.

\begin{theorem}\label{log}
	The value function of  \eqref{dp1} is given by
	\begin{align}\label{valuefunction1}
		\widehat V(t,x;w)=(x-w)^2e^{-\rho^2(T-t)}+\dfrac{\lambda\rho^2}{4}(T^2-t^2)-\dfrac{\lambda}{2}\left(\rho^2T+\log\dfrac{\lambda\Vert h' \Vert_2^2}{2e\sigma^2}\right)(T-t)-(w-z)^2,
	\end{align}
	and the corresponding optimal control process is $\widehat\Pi^*$ with quantile function
	\begin{align}\label{quantilefunction1}
	Q_{\widehat\Pi^*}(p)&=-\dfrac{\rho}{\sigma}(x-w)+\sqrt{\dfrac{\lambda}{2\sigma^2\Vert h'\Vert^2_2}}e^{\frac{1}{2}\rho^2(T-t)}h'(1-p).
	\end{align}
	Moreover, the mean and variance of $\Pi^*$ are
	\begin{align*}
		\mu(\widehat\Pi^*) =-\dfrac{\rho}{\sigma}(x-w),~~and ~~
	\sigma(\widehat\Pi^*)^2=\dfrac{\lambda}{2\sigma^2}e^{\rho^2(T-t)}.
	\end{align*}
	The optimal wealth process under $\Pi^*$ is the unique solution of the SDE
	\begin{align*}
		\begin{split}
		\d X_t^*&=-\rho^2(X_t^*-w)\d t+\sqrt{\rho^2(X_t^*-w)^2+\dfrac{\lambda}{2}e^{\rho^2(T-t)}}\d W_t
	    \end{split}
	\end{align*}
	with     $X_0^*=x_0$. 
	Finally, the Lagrange multiplier $w$ is given by $$w=\frac{ze^{\rho^2T}-x_0}{e^{\rho^2T}-1}.$$
\end{theorem}
\begin{remark}
By \eqref{quantilefunction1},  we can see that  the optimal exploratory distribution is  also   uniquely determined   by $h'$.   Since the form of  $\log\Phi_h$  affects  the value function,   even though the form of optimal distributions is the same, it is to be expected that  the variance of the optimal distributions is different from \eqref{quantilefunction}. 
It is worth  pointing  that  the mean and variance of the optimal distributions  are the same as those in \cite{WZ20} where the differential  entropy is used as a regularizer, which is an interesting observation.  This is because 
for the  payoff
 function depending only on the mean and variance processes of the distributional control,  the Gaussian distribution maximizes the entropy when the mean and variance  are fixed, and  the maximized
MV constrained entropy and $\log\Phi_h$  are equal  and  both logorithmic in the given standard deviation   and independent of the mean.  
 Moreover, since different $h$ corresponds to different  exploratory distributions, our  optimal exploratory distributions are no longer necessarily Gaussian as in  \cite{WZ20}, and are  dictated by the choice of Choquet regularizers, which can be such as Gaussian, uniform distribution  or exponential distribution. 
\end{remark}

Parallel to  Example \ref{exm:3.4}, we  give Example \ref{exm:4.3}. Theorem \ref{log} yields that both  the mean  and the variance of the  optimal distribution are independent of $h$,  so we will specify only  its quantile function. 
\begin{example}\label{exm:4.3}
(i) Let $h(p)=-p\log(p)$. Then we have
    $$\log\Phi_h (\Pi)=\log\int_0^\infty  \Pi([x,\infty)) \log(\Pi([x,\infty)))\d x.$$    The  optimal  policy is   a shifted-exponential  distribution   given as
$$\Pi^*(u; t,x)=1-\exp\left\{-\sqrt{\frac {2\sigma^2}{\lambda e^{\rho^2(T-t)}}}(u+\frac{\rho}{\sigma}(x-w)) -1\right\}.$$ 

(ii) Let $h(p)=\int_0^p z(1-s)\d s$, where $z$ is the standard normal quantile function. We have  $\log \Phi_h (\Pi) =\log\int_0^1 Q_\Pi(p) z (p) \d p$.  The  optimal  policy  is a normal   distribution  given by
$$
{\Pi}^{*}(\cdot ; t,x)=  {\mathrm N}\left(-\dfrac{\rho}{\sigma}(x-w),\dfrac{\lambda}{2\sigma^2}e^{\rho^2(T-t)}\right).
$$ 

(iii)
Let $h(p)=p-p^2$. Then  $
 \log \Phi_h(\Pi) =  \log \E[|X_1-X_2|]-\log2$.  The  optimal  policy  $\Pi^*(\cdot;x)$ is  a uniform  distribution   given as
  $$\mathrm{U}\left[-\dfrac{\rho}{\sigma}(x-w)-\sqrt{\dfrac{3\lambda}{2\sigma^2}e^{\rho^2(T-t)}},-\dfrac{\rho}{\sigma}(x-w)+\sqrt{\dfrac{3\lambda}{2\sigma^2}e^{\rho^2(T-t)}}\right].$$

\end{example}

 Next, we consider the solvability equivalence between the classical and the exploratory MV problems. Here,   “solvability equivalence”  implies  that the solution of one
 problem will lead to that of the other directly, without needing to solve it separately.   
 Recall  the classical MV problem in Section \ref{sec:2.2}.   The explicit forms of optimal control  and value function, denoted respectively by $u^*$ and $V^{cl}$,  were given by Theorem 3.2-(b) of \cite{WZ20}. We provide the solvability equivalence between the  classical and the exploratory MV problems defined by  \eqref{object}, \eqref{vf} and  \eqref{dp1}, respectively.  Since the proof is similar to that  of Theorem 9 in Appendix C of \cite{WZZ20a},   we omit the details here.

\begin{proposition}\label{eq}
	The following three statements (a), (b),  (c) are equivalent.\\
	(a) The function $V(t,x;w)=(x-w)^2e^{-\rho^2(T-t)}-\dfrac{\lambda^2\Vert h'\Vert_2^2}{4\rho^2\sigma^2}(e^{\rho^2(T-t)}-1)-(w-z)^2$, $(t,x)\in [0,T]\times\mathbb R$, is the value function of the EMV problem \eqref{vf} and the optimal feedback control is $\Pi^*$, whose quantile function is
	\begin{align*}
		Q_{\Pi^*}(p)&=-\dfrac{\rho}{\sigma}(x-w)+\dfrac{\lambda h'(1-p)}{2\sigma^2}e^{\rho^2(T-t)}.
	\end{align*}\\
	(b) The value function $	\widehat V(t,x;w)=(x-w)^2e^{-\rho^2(T-t)}+\dfrac{\lambda\rho^2}{4}(T^2-t^2)-\dfrac{\lambda}{2}(\rho^2T+\log\dfrac{\lambda\Vert h' \Vert_2^2}{2e\sigma^2})(T-t)-(w-z)^2
$, $(t,x)\in [0,T]\times\mathbb R$, is the value function of the EMV problem \eqref{dp1} and the optimal feedback control is $\widehat\Pi^*$, whose quantile function is
	\begin{align*}
	Q_{\widehat\Pi^*}(p)&=-\dfrac{\rho}{\sigma}(x-w)+\sqrt{\dfrac{\lambda}{2\sigma^2\Vert h'\Vert^2_2}}h'(1-p)e^{\frac{1}{2}\rho^2(T-t)}.	
	\end{align*}\\
	(c) The function $V^{cl}(t,x;w)=(x-w)^2e^{-\rho^2(T-t)}-(w-z)^2$, $(t,x)\in [0,T]\times\mathbb R$, is the value function of the classical MV problem \eqref{object} and the optimal feedback control is
	\begin{align*}
		u^*(t,x;w)=-\dfrac{\rho}{\sigma}(x-w).
	\end{align*}
	Moreover, the three problems above all have the  same Lagrange multiplier $$w=\frac{ze^{\rho^2T}-x_0}{e^{\rho^2T}-1}.$$
\end{proposition}

 From the  proposition  above,  we naturally want to explore more connections between $(a)$, $(b)$ and $(c)$. In fact, they have the following convergence property. 
\begin{proposition}
	Suppose that statement (a) or (b) or (c) of Proposition \ref{eq} holds. Then for each $(t,x,w)\in [0,T]\times\mathbb R \times\mathbb R$,
	\begin{align*}
	\lim\limits_{\lambda\rightarrow 0}\widehat\Pi^*(\cdot ;t,x;w)=	\lim\limits_{\lambda\rightarrow 0}\Pi^*(\cdot ;t,x;w)=\delta_{u^*(t,x;w)}(\cdot)\ \ \text{weakly},
	\end{align*}
	and  
	\begin{align*}
		\lim\limits_{\lambda\rightarrow 0}|V(t,x;w)-V^{cl}(t,x;w)|=0, ~~\text{and}~~\lim\limits_{\lambda\rightarrow 0}|\widehat V(t,x;w)-V^{cl}(t,x;w)|=0.
	\end{align*}
\end{proposition}
\begin{proof}
	The weak convergence is obvious and the convergence of value function follows from
	\begin{align*}
		\lim\limits_{\lambda\rightarrow 0} \dfrac{\lambda^2\Vert h'\Vert_2^2}{4\rho^2\sigma^2}(e^{\rho^2(T-t)}-1)=0,
	~~~~\text{and}~~~~
		\lim\limits_{\lambda\rightarrow 0} \dfrac{\lambda}{2}\log\dfrac{\lambda\Vert h'\Vert_2^2}{2e\sigma^2}=0.
\end{align*}\qed
\end{proof}
Next,  we examine the ``cost of  exploration" -- the loss in the original (i.e.,
non-regularized) objective due to exploration, which was originally defined and derived in \cite{WZZ20a} for problems with entropy regularization.  
Due to the explicit inclusion of exploration in the objectives \eqref{vf} and \eqref{dp1}, the cost of the EMV problems are defined as
\begin{align}\label{cost1}
	C^{u^*,\Pi^*}(0,x_0;w)=\left(V(0,x_0;w)+\lambda \mathbb E\left[\int_0^T\Phi_h(\Pi_t^*)dt|X_0^{\Pi^*}=x_0\right]\right)-V^{cl}(0,x_0;w),
\end{align}
and
\begin{align}\label{cost2}
	\widehat C^{u^*,\widehat\Pi^*}(0,x_0;w)=\left(\widehat V(0,x_0;w)+\lambda \mathbb E\left[\int_0^T\log\Phi_h(\widehat\Pi_t^*)dt|X_0^{\widehat \Pi^*}=x_0\right]\right)-V^{cl}(0,x_0;w).
\end{align}

\begin{proposition}\label{costtheorem}
	Suppose that statement (a) or (b) or (c) of Proposition \ref{eq} holds.  Then the cost of exploration for the EMV problem are, respectively, given as
	\begin{align}\label{cost01}
		C^{u^*,\Pi^*}(0,x_0;w)=\dfrac{\lambda^2\Vert h'\Vert_2^2}{4\rho^2\sigma^2}(e^{\rho^2T}-1),
	\end{align}
	and
\begin{align}\label{cost02}
	\widehat C^{u^*,\widehat\Pi^*}(0,x_0;w)=\dfrac{\lambda T}{2}~.
\end{align}

\end{proposition}
\begin{proof}
Note  that $$\Phi_h(\Pi^*_t)=\sigma(\Pi^*_t)\Vert h'\Vert_2=\dfrac{\lambda\Vert h'\Vert_2^2}{2\sigma^2}e^{\rho^2(T-t)},$$ and $$\log\Phi_h(\widehat\Pi^*_t)=\log\left(\sigma(\widehat\Pi^*_t)\Vert h'\Vert_2\right)=\frac{1}{2}\log\left(\dfrac{\lambda \Vert h'\Vert^2_2}{2\sigma^2}e^{\rho^2(T-t)}\right).$$ 
Bringing $\Phi_h(\Pi^*_t)$ and   $\log\Phi_h(\widehat\Pi^*_t)$ back into \eqref{cost1} and \eqref{cost2},  respectively, we can get \eqref{cost01} and \eqref{cost02}.\qed
\end{proof}
\begin{remark}
	The costs of exploration for the two EMV problems   are quite different.   When $\Phi_h$ is regarded as the regularizer, the derived exploration cost    does depend on the unknown model parameters  through $h$, $\mu$ and $\sigma$. \eqref{cost01} implies that, with other parameters being equal, to reduce the exploration cost one should choose regularizers with smaller  values of $\|h'\|_2$. 
 Moreover, by \eqref{variance}, we have 
	$$C^{u^*, \Pi^*}(0,x_0;w)=\frac{\lambda\|h'\|_2 }{2 \rho^2 }\sigma^*(x_0)-\dfrac{\lambda^2\Vert h'\Vert_2^2}{4\rho^2\sigma^2}, $$ meaning that the cost is proportional to the standardized deviation of the exploratory control, 	but  inversely proportional to  the square of the Sharp ratio $\rho^2$. 
In contrast,    when $\log\Phi_h$ is regarded as the regularizer,  the derived exploration cost only depends on $\lambda$ and $T$.    It is also interesting to note that 	$\widehat C^{u^*,\Pi^*}(0,x_0;w)$ in \eqref{cost02} is the same as the one using DE as the regularizer; see Theorem 3.4 of  \cite{WZ20}. 

Nevertheless, they also have some common features. 	The exploration cost  increases as  the exploration weight $\lambda$ and  the exploration horizon $T$  increase, due to more emphasis placed on exploration. In addition, the costs are both  independent of the Lagrange multiplier, which   suggests that the exploration cost will not increase when the agent is more aggressive (or risk-seeking) reflected by the expected target $z$ or equivalently the Lagrange multiplier $w$.
	\end{remark}

\begin{remark}	\label{costremark}	 
To compare   $C^{u^*,\Pi^*}(0,x_0;w)$ and $\widehat C^{u^*,\Pi^*}(0,x_0;w)$,  we have
	\begin{align*}
		\dfrac{C^{u^*,\Pi^*}(0,x_0;w)}{\widehat C^{u^*,\widehat\Pi^*}(0,x_0;w)}=\dfrac{\lambda\Vert h'\Vert_2^2}{2\sigma^2}\dfrac{e^{\rho^2T}-1}{\rho^2T}=\dfrac{\lambda\Vert h'\Vert_2^2}{2\sigma^2}\left(1+\sum\limits_{n=1}^{\infty}\dfrac{\rho^{2n}T^n}{(n+1)!}\right).
	\end{align*}
	Then we can easily verify which regularizer has smaller exploration cost under determined market parameters.  In general, from a cost point of view, when $\lambda$, $\Vert h'\Vert_2$  and $\rho^2$ are small enough and $\sigma$ is relatively large,  $\Phi_h$ is a good choice to reduce cost;  otherwise $\log \Phi_h$ may be a better choice.

	\end{remark}
\section{RL ALGORITHM DESIGN}\label{sec:5}
\subsection{Policy improvement}

In RL setting, the policy improvement is an important process which ensures the existence of a new policy better than any given policy. In Proposition \ref{prop:EMV_prob},  we have showed that the EMV problem in \eqref{vf} can be maximized within a location–scale family of distributions. Such a property  is  also applied to the   EMV problem in \eqref{dp1} when  $\log\Phi_h$ is regarded as the regularizer.
In  the following theorem, by Itô's formula,  we can also verify that for any given policy, when the regularizer is $\Phi_h$ or  $\log\Phi_h$, there always exists a better policy in a  location-scale family which depends on $h$. So we can search the optimal exploration distribution only in this location-scale family.

\begin{theorem}\label{pi}
	Let $w\in\mathbb R$ be fixed and $\Pi~(resp.~\widehat \Pi)$ be an arbitrarily given admissible feedback control whose corresponding value function is $V^{\Pi}(t,x;w)~(resp.~\widehat V^{\Pi}(t,x;w))$ under regularizer $\Phi_h~(resp.~\log\Phi_h)$. Suppose that $V^{\Pi}(t,x;w)~(resp.~\widehat V^{\Pi}(t,x;w))\in C^{1,2}([0,T)\times\mathbb R\cap C^0([0,T]\times\mathbb R))$ and $V^{\Pi}_{xx}(t,x;w)~(resp.~\widehat V^{\Pi}_{xx}(t,x;w))>0$ for any $(t,x)\in[0,T)\times\mathbb R$. Suppose further that the feedback control $\widetilde{\Pi}~(resp.~\widetilde{\widehat {\Pi}})$ whose quantile function is
	\begin{align}\label{pit}
		Q_{\widetilde{\Pi}}(p)&=-\dfrac{\rho}{\sigma}\dfrac{V_x^{\Pi}}{V_{xx}^{\Pi}}+\dfrac{\lambda}{\sigma^2V_{xx}^{\Pi}}h'(1-p)\\
		resp.~Q_{\widetilde{\widehat {\Pi}}}(p)&=-\dfrac{\rho}{\sigma}\dfrac{\widehat V_x^{\widehat\Pi}}{\widehat V_{xx}^{\widehat\Pi}}+\sqrt{\dfrac{\lambda}{\sigma^2\Vert h'\Vert^2_2\widehat V_{xx}^{\widehat\Pi}}}h'(1-p)\label{pit2}
	\end{align}
	is admissible. Then
	\begin{align*}
		V^{\widetilde{\Pi}}(t,x;w)&\leqslant V^{\Pi}(t,x;w),\ \ (t,x)\in [0,T)\times\mathbb R,\\
		resp.~\widehat V^{\widetilde{\widehat\Pi}}(t,x;w)&\leqslant \widehat V^{\widehat\Pi}(t,x;w),\ \ (t,x)\in [0,T)\times\mathbb R.
	\end{align*}
\end{theorem}
\begin{proof}
	Let $\widetilde{\Pi}=\{\widetilde{\Pi}_s,s\in[t,T]\}$ and $\widetilde{\widehat\Pi} =\{\widetilde{\widehat\Pi} _s,s\in[t,T]\}$ be the open-loop control generated by the given feedback control policies $\widetilde{\Pi}$ and $\widetilde{\widehat\Pi}$, respectively. By assumption, ${\widetilde{\Pi}}$ and $\widetilde{\widehat\Pi}$ are admissible. Applying Itô's formula, we have for any $(t,x)\in [0,T]\times \mathbb R$,
	\begin{align}\label{ito}
		\begin{split}
		V^{\Pi}(s,X_s^{\widetilde{\Pi}})&=V^{\Pi}(t,x)+\int_t^s
		V_t^{\Pi}(v,X_v^{\widetilde{\Pi}})\d v+\int_t^s V_x^{\Pi}(v,X_v^{\widetilde{\Pi}})\d X_v^{\widetilde{\Pi}}\\
		&\quad+\dfrac{1}{2}\int_t^sV_{xx}^{\Pi}(v,X_v^{\widetilde{\Pi}})\d<X^{\widetilde{\Pi}},X^{\widetilde{\Pi}}>_v\\
		&=V^{\Pi}(t,x)+\int_t^sV_x^{\Pi}(v,X_v^{\widetilde{\Pi}})\sigma\sqrt{\mu(\widetilde{\Pi}_v)^2+\sigma(\widetilde{\Pi}_v)^2}\d W_v\\
		&\quad+\int_t^s[V_t^{\Pi}(v,X_v^{\widetilde{\Pi}})+\rho\sigma\mu(\widetilde{\Pi}_v)V_x^{\Pi}(v,X_v^{\widetilde{\Pi}})+\dfrac{\sigma^2}{2}(\mu(\widetilde{\Pi}_v)^2+\sigma(\widetilde{\Pi}_v)^2)V_{xx}^{\Pi}(v,X_v^{\widetilde{\Pi}})]\d v.
		\end{split}
	\end{align}
	 Let $\tau_n:=\inf\{s\geqslant t:\int_t^s \sigma^2 V_x^{\Pi}(v,X_v^{\widetilde{\Pi}})^2(\mu(\widetilde{\Pi}_v)^2+\sigma(\widetilde{\Pi}_v)^2)\d v\geqslant n\}$ be a family of stopping times, then substituting $s\wedge\tau_n$ into \eqref{ito} and taking expectation we get
	 \begin{align}\label{ex}
	 	\begin{split}
	 		V^{\Pi}(t,x)&=\mathrm E\left[ V^{\Pi}(s\wedge\tau_n,X_{s\wedge\tau_n}^{\widetilde{\Pi}})-\int_t^{s\wedge\tau_n}[V_t^{\Pi}(v,X_v^{\widetilde{\Pi}})+\rho\sigma\mu(\widetilde{\Pi}_v)V_x^{\Pi}(v,X_v^{\widetilde{\Pi}})\right.\\
	 	&\quad+\left.\dfrac{\sigma^2}{2}(\mu(\widetilde{\Pi}_v)^2+\sigma(\widetilde{\Pi}_v)^2)V_{xx}^{\Pi}(v,X_v^{\widetilde{\Pi}})]\d v|X_t^{\widetilde{\Pi}}=x \right].
	 	\end{split}
	 \end{align}
	 On the other hand, by standard argument we have
	 \begin{align*}
	 	V_t^{\Pi}(t,x)+\rho\sigma\mu(\Pi)V_x^{\Pi}(t,x)+\dfrac{\sigma^2}{2}(\mu(\Pi)^2+\sigma(\Pi)^2)V_{xx}^{\Pi}(t,x)-\lambda\Phi_h(\Pi)=0.
	 \end{align*}
	 It follows that
	 \begin{align}\label{sa}
	 	V_t^{\Pi}(t,x)+\min\limits_{\Pi'\in\mathscr{P}(\mathbb R)}\left[\rho\sigma\mu(\Pi')V_x^{\Pi}(t,x)+\dfrac{\sigma^2}{2}(\mu(\Pi')^2+\sigma(\Pi')^2)V_{xx}^{\Pi}(t,x)-\lambda\Phi_h(\Pi')\right]\leqslant 0.
	 \end{align}
	 By \eqref{ms}, we know $\widetilde{\Pi}$ is the minimizer of \eqref{sa}. Substituting $\widetilde{\Pi}$ into \eqref{sa} and bringing back to \eqref{ex} we have
	 \begin{align}\label{ts}
	 	V^{\Pi}(t,x)\geqslant\mathrm E\left[V^{\Pi}(s\wedge\tau_n,X_{s\wedge\tau_n}^{\widetilde{\Pi}})-\int_t^{s\wedge\tau_n}\lambda\Phi_h(\widetilde{\Pi}_v)dv|X_t^{\widetilde{\Pi}}=x\right].
	 \end{align}
	 Taking $s=T$ in \eqref{ts} and sending $n$ to $\infty$, we obtain 
	 \begin{align*}
	 	V^{\Pi}(t,x)\geqslant \mathrm E\left[V^{\widetilde{\Pi}}(T,X_T^{\widetilde{\Pi}})-\lambda\int_t^T\Phi_h(\widetilde{\Pi}_v)dv|X_t^{\widetilde{\Pi}}=x\right]=V^{\widetilde{\Pi}}(t,x).
	 \end{align*}
	 The proof of regularizer $\log\Phi_h$ is almost the same, so we omit it.\qed
\end{proof}

\begin{theorem}\label{update}
	Let $\Pi^0(u;t,x,w)$ be a feedback control which has quantile function
	\begin{align}\label{initialdistribution}
		Q_{\Pi^0}(p)=Q_{\widehat\Pi^0}(p)=a(x-w)+c_1e^{c_2(T-t)}h'(1-p),
	\end{align}
	and $\{\Pi^n(u;t,x,w)\}$  and  $\{\widehat\Pi^n(u;t,x,w)\}$ be the sequence of feedback controls updated by \eqref{pit} and \eqref{pit2}, respectively. Denoted by $\{V^{\Pi^n}(t,x;w)\}$  and $\{\widehat V^{\widehat\Pi^n}(t,x;w)\}$ the sequence of corresponding value functions. Then
	\begin{align*}
		\lim\limits_{n\rightarrow\infty}\Pi^n(\cdot;t,x,w)&=\Pi^*(\cdot;t,x,w)\ \ 
		\text{weakly},\\
resp.~\lim\limits_{n\rightarrow\infty}\widehat\Pi^n(\cdot;t,x,w)&=\widehat\Pi^*(\cdot;t,x,w)\ \ 
		\text{weakly},	\end{align*}
	and
	\begin{align*}
		\lim\limits_{n\rightarrow\infty}V^{\Pi^n}(t,x;w)&=V(t,x;w),\ \ (t,x)\in [0,T),\\resp.~\lim\limits_{n\rightarrow\infty}\widehat V^{\widehat\Pi^n}(t,x;w)&=\widehat V(t,x;w),\ \ (t,x)\in [0,T)\times\mathbb R,
	\end{align*}
	for any $(t,x,w)\in[0,T]\times\mathbb R\times\mathbb R$, where $\Pi^*$ and $\widehat\Pi^*$  in \eqref{quantilefunction} and \eqref{quantilefunction1} are the optimal controls,  and  $V$ and $\widehat V$ are the value functions given by  \eqref{valuefunction} and \eqref{valuefunction1}.	
\end{theorem}
\begin{proof}  Here we only provide the detailed proof for the case of $\Phi_h$, and the results of $\log\Phi_h$ can be derived in the same way. 
Let $\{\Pi^0_s\}$ be the open-loop control generated by $\Pi^0$. We can verify that $\{\Pi^0_s\}$ is admissible.
	The dynamic of wealth under $\Pi^0$ is
	\begin{align*}
		\d X_t^{\Pi^0}=\rho\sigma\mu(\Pi^0)\d t+\sigma\sqrt{\mu(\Pi^0)^2+\sigma(\Pi^0)^2}\d W_t,\ \ X_t^{\Pi^0}=x,
	\end{align*}
	and the value function under $\Pi^0$ is
	\begin{align*}
		V^{\Pi^0}(t,x)=\mathrm E\left[\int_t^T-\lambda\Phi_h(\Pi_v^0)dv+(X_T^{\Pi^0}-w)^2|X_t^{\Pi^0}=x\right]-(w-z)^2.
	\end{align*}
	By Feynman–Kac formula, we deduce that $V^{\Pi^0}$ satisfies the following PDE
	\begin{align*}
		V_t(t,x)+\rho\sigma\mu(\Pi^0)V_x(t,x)+\dfrac{1}{2}\sigma^2(\mu(\Pi^0)^2+\sigma(\Pi^0)^2)V_{xx}(t,x)-\lambda\Phi_h(\Pi^0)=0,
	\end{align*}
	with terminal condition $V^{\Pi^0}(T,x)=(x-w)^2-(w-z)^2$. Solving this equation we obtain
	\begin{align*}
		V^{\Pi^0}(t,x;w)=(x-w)^2e^{(2\rho\sigma a+\sigma^2a^2)(T-t)}	+F_0(t),
	\end{align*}
	where $F_0(t)$ is a smooth function which only depends on $t$. Obviously, $V^{\Pi^0}(t,x;w)$ satisfies the conditions of Theorem \ref{pi}, so we can use \eqref{pit} to obtain $\Pi^1$ whose quantile function is	\begin{align*}
		Q_{\Pi^1}(p)&=-\dfrac{\rho}{\sigma}(x-w)+\dfrac{\lambda h'(1-p)}{2\sigma^2e^{(2\rho\sigma a+\sigma^2a^2)(T-t)}},
	\end{align*}	with \begin{align*}
		\mu(\Pi^1)=-\dfrac{\rho}{\sigma}(x-w),~~~\text{and}~~~
		\sigma^2(\Pi^1)=\dfrac{\lambda^2\Vert h'\Vert_2^2}{4\sigma^2e^{2(2\rho\sigma a+\sigma^2 a^2)(T-t)}}.
	\end{align*}
	By repeating the above program with $\Pi^1$, we have
	\begin{align*}
		V^{\Pi^1}(t,x;w)=(x-w)^2e^{-\rho^2(T-t)}	+F_1(t),
	\end{align*}
	where $F_1(t)$ is a smooth function which only depends on $t$. Using Theorem \ref{pi} again we obtain $\Pi^2$ whose quantile function is
	\begin{align*}
		Q_{\Pi^2}(p)=-\dfrac{\rho}{\sigma}(x-w)+\dfrac{\lambda h'(1-p)}{2\sigma^2}e^{\rho^2(T-t)},	\end{align*} with 
		\begin{align*}	\mu(\Pi^2)=-\dfrac{\rho}{\sigma}(x-w),~~~~\text{and}~~~~
		\sigma^2(\Pi^2)=\dfrac{\lambda^2\Vert h'\Vert_2^2}{4\sigma^4}e^{2\rho^2(T-t)}.
	\end{align*}
	By \eqref{quantilefunction}-\eqref{variance}, we know that $\Pi^2$ is optimal.\qed
\end{proof}
The above theorem shows that when designing a RL algorithm, the distribution with the quantile form \eqref{initialdistribution} can be selected as the initial distribution to ensure the convergence.

\subsection{The EMV algorithm}
In this section, we aim to solve \eqref{vf} and \eqref{dp1} by assuming that  there is no knowledge about the underlying parameters. One method to overcome this problem is to replace the parameters by their estimations. However, as mentioned in Introduction, the estimations are usually very sensitive to the sample.  We will give an offline RL algorithm based on the Actor-Critic algorithm in \cite{KT99}, \cite{SB18}  and \cite{JZ22b}. The Actor-Critic algorithm is essentially a policy-based algorithm, but additionally learns the value function in order to help the policy function learn better. Meanwhile, we use a self-correcting scheme in \cite{WZ20} to learn the Lagrange multiplier $w$. 

 Here, we only  present the RL algorithm for the case of  $\Phi_h$ to solve \eqref{vf}.  
When using $\log\Phi_h$ as the regularizer, we  only  need to replace $\Phi_h$ by $\log\Phi_h$ and  modify the parameterization appropriately.

In continuous-time setting, we first discretize $[0,T]$ into $N$ small intervals $[t_i,t_{i+1}], (i=0,1,...,N-1)$ whose length is equal to ${T}/{N}=\Delta t$. We use policy gradient principle to update Actor; and for Critic, \cite{JZ22a} showed that the time-discretized algorithm converges as $\Delta t \rightarrow 0$ as long as the corresponding discrete-time algorithms converges, thus we adopt a learning approach of temporal difference error (the TD error; see \cite{D20} and \cite{WZ20}).  
Assume that  $\Pi $ is a given admissible feedback policy and let $\mathscr D=\{(t_i,x_{t_i}),i=0,1,...,N\}$ be a set of samples, the initial sample is $(0,x_0)$, then for $i=1,2,...,N$, we sample $u_{t_{i-1}}$ from $\Pi_{t_{i-1}}$ and get $x_{t_i}$ at $t_i$. 

On the one hand, we have
\begin{align*}
	V^{\Pi }(t,x)=\mathrm{E}\left[(X_T^{\Pi }-w)^2-\lambda\int_t^T\Phi_h(\Pi_s)ds|X_t^{\Pi }=x\right]-(w-z)^2,
\end{align*}
so the TD error at $t_i$ is
\begin{align*}
	\delta_i=-\lambda\Phi_h(\Pi_{t_i})\Delta t+V^{\Pi }(t_{i+1},X_{t_{i+1}})-V^{\Pi }(t_{i},X_{t_{i}}),\quad i = 0,1,...,N-1.
\end{align*}

On the other hand, based on \eqref{valuefunction}, we can parameterize the Critic value by
\begin{align*}
	V^{\theta}(t,x)=(x-w)^2e^{-\theta_2(T-t)}-\theta_1e^{\theta_0(T-t)}-(w-z)^2.
\end{align*}
For a single point $t_i$, we define the loss function as 
\begin{align}\label{lossfunction}
	L(\theta)=\dfrac{1}{2}(U_{t_i}-V^{\theta}(t_{i},X_{t_{i}}))^2,
\end{align}
where $U_{t_i}$ is the estimation of $V(t_{i},X_{t_{i}})$. We take $U_{t_i}$ a bootstrapping estimate $-\lambda\Phi_h(\Pi_{t_i})\Delta t+V^{\theta}(t_{i+1},X_{t_{i+1}})$ in \eqref{lossfunction} as the temporal difference target which will not generate gradient to update the value function automatically. So the gradient of the loss function is
\begin{align}\label{lossgradient}
	\nabla_{\theta}L(\theta)=-(-\lambda\Phi_h(\Pi_{t_i})\Delta t+V^{\theta}(t_{i+1},X_{t_{i+1}})-V^{\theta}(t_{i},X_{t_{i}}))\nabla_{\theta}V^{\theta}(t_i,X_{t_i}).
\end{align}
Let $\alpha_{\theta}$ be the learning rate of $\theta$, then by \eqref{lossgradient}, we can get the gradient and the update rule of $\theta$ with a set of sample $\mathscr D$:
\begin{align}\label{thetagradient}
	\nabla\theta =-\sum\limits_{i=0}^{N-1}\dfrac{\partial V^{\theta}}{\partial \theta}(t_i,x_{t_i})[V^{\theta}(t_{i+1},x_{t_{i+1}})-V^{\theta}(t_{i},x_{t_{i}})-\lambda \Phi_h(\Pi_{t_i}^{\phi})\Delta t],
\end{align} and 
\begin{align}\label{thetaupdate}
	\theta \longleftarrow \theta - \alpha_{\theta}\nabla\theta.
\end{align}
Based on Theorem \ref{update}, we can parameterize the policy by $\Pi ^{\phi}$ with quantile function 
\begin{align*}
	Q_{\Pi_t^{\phi}}(p)=-\phi_0(x-w)+e^{\frac{1}{2	}\phi_1+\frac{1}{2}\phi_2(T-t)}h'(1-p).
\end{align*}
By Lemma 2.3 of  \cite{HWZ23}, we know  that
\begin{align*}
	\Phi_h(\Pi_t^{\phi})=\displaystyle\int_0^1(-\phi_0(x-w)+e^{\frac{1}{2	}\phi_1+\frac{1}{2}\phi_2(T-t)}h'(p)^2)\d p.
\end{align*}
Let  $g(t,x;\phi)=\nabla_{\theta}V^{\Pi ^{\phi}}(t,x)$ be the policy gradient of $\Pi ^{\phi}$ and $p(t,\phi)=\Phi_h(\Pi_t^{\phi})$,  together with Theorem 5 of  \cite{JZ22b},  $g(t,x;\phi)$ has the following representation:
\begin{align}\label{policygradient}
	g(t,x;\phi)=\mathrm E\left[\displaystyle\int_t^T\left\{\dfrac{\partial}{\partial\phi}\log\dot \Pi^{\phi}_t\left(\d V^{\Pi _{\phi}}(s,X_s^{\Pi ^{\phi}})-\lambda p(s,\phi)\d s\right)-\lambda \dfrac{\partial p}{\partial \phi}(s,\phi)\d s\right\}\left|X_t^{\Pi ^{\phi}}=x\right.\right],
\end{align}
where $\dot \Pi^{\phi}_t$ is the  density function of  $ \Pi^{\phi}_t$.  Let $\alpha_{\phi}$ be the learning rate of $\phi$, then by \eqref{policygradient}, we can also get the gradient and the update rule of $\theta$ with a set of sample $\mathscr D$:
\begin{align}\label{phigradient}
	\begin{split}
		\nabla\phi =&\sum\limits_{i=0}^{N-1}\left\{\dfrac{\partial }{\partial \phi}\right.\log\dot  \Pi^{\phi}(u_{t_i}|t_i,x_{t_i})[V^{\theta}(t_{i+1},x_{t_{i+1}})-V^{\theta}(t_{i},x_{t_{i}})-\lambda \Phi_h(\Pi_{t_i}^{\phi})\Delta t]\\
	&-\left.\lambda \dfrac{\partial p}{\partial \phi}(t_i,x_{t_i},\phi)\Delta t\right\},
	\end{split}
\end{align} and 
\begin{align}\label{phiupdate}
	\phi \longleftarrow \phi - \alpha_{\phi}\nabla\phi.
\end{align}
Let $\alpha_w$ be the learning rate of $\phi$, then by the constraint $\mathrm E[X_T]=z$ we can get the standard stochastic approximation update rule:
\begin{align*}
	w_{n+1}=w_n-\alpha_w(\dfrac{1}{m}\sum\limits_{i=j-m+1}^jx_T^{(i)} -z),
\end{align*}
where $x_T^{(i)}$ is the last point of sample $i$ and $j\equiv 0 \mod m$.

We summarize the algorithm as pseudocode in \textbf{Algorithm 1}. 
\begin{algorithm}
	\caption{Actor-Critic Algorithm for EMV Problem}	
	\textbf{Input:} initial wealth $x_0$, the parameters ($\mu, \sigma, r$) of Market, the target $z$, exploration weight $\lambda$, investment horizon $T$, time step $\Delta t$, number of time grids $N$, learning rates $\alpha_{\theta},\ \alpha_{\phi}, \alpha_w$, number of episodes $K$, sample average size $m$, and a simulator of the market called $Market$.\\
	\textbf{Learning procedure:}
	Initialize $\theta,\ \phi,\ w$.
	\begin{algorithmic}
		\FOR {episode $j=1$ \textbf{to} $K$}
		\STATE Initialize $n=0$
		\STATE  $x_{t_n} \leftarrow x_0$
		\WHILE {$n<N$}
		\STATE Compute and store $\dfrac{\partial}{\partial \theta}V^{\theta}(t_n,x_{t_n})$
		\STATE Sample $u_{t_n}$ from $\Pi^{\phi}(\cdot |t_n,x_{t_n} )$.
		\STATE Compute and store $p(t_n,x_{t_n},\phi)$.
		\STATE Compute and store $\dfrac{\partial p}{\partial\phi} (t_n,x_{t_n},\phi).$
		\STATE Compute and store $\dfrac{\partial}{\partial\phi}\log\dot \Pi^{\phi}(u_{t_n}|t_n,x_{t_n})$.
		\STATE Apply $u_{t_n}$ to the market simulator and get the state $x$ at next time point. 
		\STATE Store $x_{t_{k+1}}\leftarrow x$.
		\STATE $k\leftarrow k+1$.
		\ENDWHILE
		\STATE Store the terminal wealth $x_T^{(j)}\leftarrow x_{t_N}$.
		\STATE Compute the gradient of $\theta$ and $\phi$ by \eqref{thetagradient} and \eqref{phigradient}, respectively. 
		\begin{align*}
			\nabla\theta =-\sum\limits_{i=0}^{N-1}\dfrac{\partial V^{\theta}}{\partial \theta}(t_i,x_{t_i})[V^{\theta}(t_{i+1},x_{t_{i+1}})-V^{\theta}(t_{i},x_{t_{i}})-\lambda p(t_i,x_{t_i},\phi)\Delta t],
		\end{align*}
		and
		\begin{align*}
			\nabla\phi =&\sum\limits_{i=0}^{N-1}\left\{\dfrac{\partial \log \dot \Pi^{\phi} }{\partial \phi}\right.(u_{t_i}|t_i,x_{t_i})[V^{\theta}(t_{i+1},x_{t_{i+1}})-V^{\theta}(t_{i},x_{t_{i}})-\lambda p(t_i,x_{t_i},\phi)\Delta t]\\
			&-\left.\lambda \dfrac{\partial p}{\partial \phi}(t_i,x_{t_i})\right\}.
		\end{align*}
		\STATE Update $\theta$ and $\phi$ by \eqref{thetaupdate} and \eqref{phiupdate}, respectively.
		\begin{align*}
			\theta \longleftarrow \theta - \alpha_{\theta}l(j)\nabla\theta
		\end{align*}
		\begin{align*}
			\phi \longleftarrow \phi - \alpha_{\phi}l(j)\nabla\phi
		\end{align*}
		\STATE Update $w$ every $m$ episodes:
		\IF {$j\equiv 0 \mod\ m$}
		\STATE $$w\leftarrow w-\alpha_w(\dfrac{1}{m}\sum\limits_{i=j-m+1}^{j}x_T^{(i)}-z).$$
		\ENDIF
		\ENDFOR
	\end{algorithmic}
\end{algorithm}
\section{SIMULATION}\label{sec:6}
In this section, we conduct simulations and test our algorithm presented in \textbf{Algorithm 1}. In our setting, we take investment horizon to be $T=1$ and time step to be $\Delta t=\frac{1}{252}$, which can be interpreted as the MV problem considered over one-year period, and then the number of time grids is $N=252$ naturally. We can take the annualized interest rate to be $r=2\%$ and take the annualized return $\mu$ and volatility $\sigma$ from $\{-50\%,\ -30\%,\ -10\%,\ 10\%,\ 30\%,\ 50\%\}$ and $\{10\%,\ 20\%,\ 30\%,\ 40\%\}$, respectively. Let the initial wealth to be $x_0=1$ and the annualized target return on the terminal wealth is $40\%$ which yields $z=1.4$.

For our algorithm, we take the number of episodes $K=20000$, and take the sample average size for Lagrange multiplier $m=10$. Based on Proposition \ref{costtheorem} and Remark \ref{costremark}, to control their exploration costs, the exploration weight $\lambda$ is taken as $0.01$ when we apply $\Phi_h$ as  the regularizer, and $0.1$ for  $\log\Phi_h$ being  the regularizer. The learning rates are taken as $\alpha_{\theta}=\alpha_{\phi}=\alpha_{w}=0.01$ with decay rate $l(j)=j^{-0.51}$. 

Based on Examples \ref{exm:3.4} and \ref{exm:4.3}, we mainly investigate the simulation results for three  exploration distributions:  Gaussian, exponential distribution and uniform distribution. We present the mean and the variance of the last 200 terminal wealth, and the corresponding Sharpe ratio $(\frac{mean-1}{\sqrt{variance}})$. The simulation results of our algorithm are presented in Tables \ref{table:normal}--\ref{table:uniform}.

\begin{table}[ht!]
\centering
\caption{Performance of  Gaussian with  $h(p)=\int_0^p z(1-s)\d s$}\label{table:normal}
	\begin{tabular}{cc|ccc|ccc}
		\toprule
		\multicolumn{1}{c}{\multirow{2}{*}{$\mu$}}&\multicolumn{1}{c|}{\multirow{2}{*}{$\sigma$}}&\multicolumn{3}{c|}{$\Phi_h$}&\multicolumn{3}{c}{$\log\Phi_h$}\\ \cline{3-8}
		 & & Mean & Variance & Sharpe ratio & Mean & Variance & Sharpe ratio\\ \hline
		-0.5& 0.1 & 1.4052 & 0.0035 & 6.8192 & 1.4052 & 0.0037 & 6.6520 \\
		-0.3& 0.1 & 1.4141 & 0.0103 & 4.0852 & 1.4143 & 0.0104 & 4.0554 \\
		-0.1& 0.1 & 1.4479 & 0.1104 & 1.3482 & 1.4485 & 0.1107 & 1.3482 \\
		0.1 & 0.1 & 1.3966 & 0.2516 & 0.7906 & 1.3970 & 0.2571 & 0.7828 \\
		0.3 & 0.1 & 1.4052 & 0.0408 & 2.0043 & 1.4055 & 0.0441 & 1.9307 \\
		0.5 & 0.1 & 1.4007 & 0.0247 & 2.5722 & 1.4007 & 0.0267 & 2.4519 \\
		
		-0.5& 0.2 & 1.4078 & 0.0147 & 3.3654 & 1.4077 & 0.0153 & 3.2939 \\
		-0.3& 0.2 & 1.4208 & 0.0458 & 1.9668 & 1.4209 & 0.0464 & 1.9534 \\
		-0.1& 0.2 & 1.4557 & 0.5046 & 0.6416 & 1.4552 & 0.5038 & 0.6413 \\
		0.1 & 0.2 & 1.3576 & 0.8506 & 0.3878 & 1.3575 & 0.8643 & 0.3846 \\
		0.3 & 0.2 & 1.3967 & 0.1402 & 1.0595 & 1.3966 & 0.1487 & 1.0284 \\
		0.5 & 0.2 & 1.3943 & 0.0739 & 1.4506 & 1.3941 & 0.0799 & 1.3945 \\
		
		-0.5& 0.3 & 1.4118 & 0.0368 & 2.1456 & 1.4117 & 0.0382 & 2.1053 \\
		-0.3& 0.3 & 1.4290 & 0.1201 & 1.2362 & 1.4292 & 0.1221 & 1.2282 \\
		-0.1& 0.3 & 1.4143 & 1.0305 & 0.4081 & 1.4126 & 1.0228 & 0.4080 \\
		0.1 & 0.3 & 1.2978 & 1.3627 & 0.2551 & 1.2974 & 1.3796 & 0.2532 \\
		0.3 & 0.3 & 1.3887 & 0.2825 & 0.7314 & 1.3884 & 0.2961 & 0.7138 \\
		0.5 & 0.3 & 1.3890 & 0.1353 & 1.0574 & 1.3886 & 0.1444 & 1.0225 \\
		
		-0.5& 0.4 & 1.4171 & 0.0761 & 1.5122 & 1.4169 & 0.0786 & 1.4872 \\
		-0.3& 0.4 & 1.4364 & 0.2507 & 0.8715 & 1.4366 & 0.2539 & 0.8665 \\
		-0.1& 0.4 & 1.3539 & 1.4238 & 0.2966 & 1.3514 & 1.4054 & 0.2965 \\
		0.1 & 0.4 & 1.2358 & 1.5370 & 0.1902 & 1.2346 & 1.5465 & 0.1887 \\
		0.3 & 0.4 & 1.3801 & 0.4691 & 0.5550 & 1.3797 & 0.4879 & 0.5436 \\
		0.5 & 0.4 & 1.3844 & 0.2119 & 0.8351 & 1.3839 & 0.2244 & 0.8103 \\
		\bottomrule
	\end{tabular}
\end{table}
\begin{table}[ht!]
\centering
\caption{Performance of  exponential distribution with  $h(p)=-p\log p$}
	\begin{tabular}{cc|ccc|ccc}
		\toprule
		\multicolumn{1}{c}{\multirow{2}{*}{$\mu$}}&\multicolumn{1}{c|}{\multirow{2}{*}{$\sigma$}}&\multicolumn{3}{c|}{$\Phi_h$}&\multicolumn{3}{c}{$\log\Phi_h$}\\ \cline{3-8}
		 & & Mean & Variance & Sharpe ratio & Mean & Variance & Sharpe ratio\\ \hline
		-0.5& 0.1 & 1.2501 & 0.0033 & 4.3463 & 1.3914 & 0.0051 & 5.4729 \\
		-0.3& 0.1 & 1.3228 & 0.0096 & 3.3001 & 1.3625 & 0.0115 & 3.3737 \\
		-0.1& 0.1 & 1.2750 & 0.0452 & 1.2934 & 1.2788 & 0.0469 & 1.2868 \\
		0.1 & 0.1 & 1.2764 & 0.1619 & 0.6867 & 1.2623 & 0.1694 & 0.6373 \\
		0.3 & 0.1 & 1.3939 & 0.0519 & 1.7287 & 1.3793 & 0.0906 & 1.2601 \\
		0.5 & 0.1 & 1.3962 & 0.0377 & 2.0408 & 1.3884 & 0.0849 & 1.3328 \\
		
		-0.5& 0.2 & 1.2590 & 0.0133 & 2.2488 & 1.3940 & 0.0204 & 2.7564 \\
		-0.3& 0.2 & 1.3274 & 0.0392 & 1.6534 & 1.3665 & 0.0473 & 1.6858 \\
		-0.1& 0.2 & 1.2027 & 0.1059 & 0.6229 & 1.1990 & 0.1049 & 0.6114 \\
		0.1 & 0.2 & 1.2645 & 0.5390 & 0.3602 & 1.2556 & 0.5358 & 0.3492 \\
		0.3 & 0.2 & 1.3791 & 0.1694 & 0.9211 & 1.3666 & 0.2282 & 0.7675 \\
		0.5 & 0.2 & 1.3856 & 0.1140 & 1.1421 & 1.3776 & 0.1887 & 0.8693 \\
		
		-0.5& 0.3 & 1.2706 & 0.0314 & 1.5271 & 1.3960 & 0.0475 & 1.8165 \\
		-0.3& 0.3 & 1.3277 & 0.0893 & 1.0964 & 1.3665 & 0.1083 & 1.1139 \\
		-0.1& 0.3 & 1.0972 & 0.0763 & 0.3521 & 1.0851 & 0.0680 & 0.3261 \\
		0.1 & 0.3 & 1.2686 & 1.0588 & 0.2610 & 1.2679 & 1.0534 & 0.2610 \\
		0.3 & 0.3 & 1.3686 & 0.3034 & 0.6691 & 1.3618 & 0.3311 & 0.6288 \\
		0.5 & 0.3 & 1.3783 & 0.1927 & 0.8616 & 1.3726 & 0.2439 & 0.7545 \\
		
		-0.5& 0.4 & 1.2846 & 0.0597 & 1.1643 & 1.3972 & 0.0879 & 1.3396 \\
		-0.3& 0.4 & 1.3138 & 0.1512 & 0.8069 & 1.3488 & 0.1828 & 0.8157 \\
		-0.1& 0.4 & 1.0129 & 0.0390 & 0.0653 & 0.9962 & 0.0390 & -0.0192\\
		0.1 & 0.4 & 1.2869 & 1.7818 & 0.2150 & 1.2973 & 1.8495 & 0.2186 \\
		0.3 & 0.4 & 1.3610 & 0.4404 & 0.5440 & 1.3614 & 0.4285 & 0.5521 \\
		0.5 & 0.4 & 1.3731 & 0.2684 & 0.7203 & 1.3711 & 0.2797 & 0.7016 \\
		\bottomrule
	\end{tabular}
\end{table}
\begin{table}[ht!]
\centering
\caption{Performance of uniform distribution with  $h(p)=p-p^2$}\label{table:uniform}
	\begin{tabular}{cc|ccc|ccc}
		\toprule
		\multicolumn{1}{c}{\multirow{2}{*}{$\mu$}}&\multicolumn{1}{c|}{\multirow{2}{*}{$\sigma$}}&\multicolumn{3}{c|}{$\Phi_h$}&\multicolumn{3}{c}{$\log\Phi_h$}\\ \cline{3-8}
		 & & Mean & Variance & Sharpe ratio & Mean & Variance & Sharpe ratio\\ \hline
		-0.5& 0.1 & 1.4057 & 0.0035 & 6.8631 & 1.4057 & 0.0038 & 6.5978 \\
		-0.3& 0.1 & 1.4077 & 0.0107 & 3.9474 & 1.4077 & 0.0109 & 3.8992 \\
		-0.1& 0.1 & 1.3663 & 0.0719 & 1.3657 & 1.3663 & 0.0722 & 1.3637 \\
		0.1 & 0.1 & 1.2843 & 0.1128 & 0.8465 & 1.2846 & 0.1160 & 0.8356 \\
		0.3 & 0.1 & 1.3873 & 0.0200 & 2.7362 & 1.3877 & 0.0240 & 2.5044 \\
		0.5 & 0.1 & 1.3953 & 0.0096 & 4.0327 & 1.3956 & 0.0137 & 3.3803 \\
		
		-0.5& 0.2 & 1.4130 & 0.0145 & 3.4269 & 1.4130 & 0.0156 & 3.3090 \\
		-0.3& 0.2 & 1.4206 & 0.0457 & 1.9682 & 1.4207 & 0.0467 & 1.9465 \\
		-0.1& 0.2 & 1.3931 & 0.3254 & 0.6892 & 1.3932 & 0.3263 & 0.6882 \\
		0.1 & 0.2 & 1.2788 & 0.4258 & 0.4272 & 1.2792 & 0.4378 & 0.4220 \\
		0.3 & 0.2 & 1.3823 & 0.0813 & 1.3407 & 1.3831 & 0.0964 & 1.2342 \\
		0.5 & 0.2 & 1.3926 & 0.0389 & 1.9901 & 1.3933 & 0.0540 & 1.6929 \\
		
		-0.5& 0.3 & 1.4215 & 0.0346 & 2.2657 & 1.4215 & 0.0369 & 2.1937 \\
		-0.3& 0.3 & 1.4363 & 0.1118 & 1.3050 & 1.4364 & 0.1140 & 1.2921 \\
		-0.1& 0.3 & 1.4274 & 0.8438 & 0.4653 & 1.4274 & 0.8459 & 0.4647 \\
		0.1 & 0.3 & 1.2795 & 0.9259 & 0.2905 & 1.2800 & 0.9460 & 0.2879 \\
		0.3 & 0.3 & 1.3803 & 0.1941 & 0.8631 & 1.3812 & 0.2239 & 0.8058 \\
		0.5 & 0.3 & 1.3914 & 0.0950 & 1.2702 & 1.3923 & 0.1257 & 1.1066 \\
		
		-0.5& 0.4 & 1.4314 & 0.0661 & 1.6786 & 1.4314 & 0.0701 & 1.6300 \\
		-0.3& 0.4 & 1.4550 & 0.2196 & 0.9711 & 1.4550 & 0.2234 & 0.9628 \\
		-0.1& 0.4 & 1.4707 & 1.7666 & 0.3542 & 1.4707 & 1.7701 & 0.3538 \\
		0.1 & 0.4 & 1.2862 & 1.6430 & 0.2233 & 1.2863 & 1.6454 & 0.2232 \\
		0.3 & 0.4 & 1.3811 & 0.3868 & 0.6127 & 1.3818 & 0.4203 & 0.5889 \\
		0.5 & 0.4 & 1.3917 & 0.1937 & 0.8899 & 1.3925 & 0.2365 & 0.8071 \\
		\bottomrule
	\end{tabular}
\end{table}

For  different values of $\mu$ and $\sigma$,  we take means of every 100 terminal wealth for different $h$ to show the tendency of the expectation of terminal wealth in Figures \ref{logtend} and  \ref{nologtend}, respectively.  
We find that the algorithm performs more significantly as $|\mu|$ increases or as $\sigma$ decreases with other parameters fixed. When $\mu<0$, exponential distribution seems to be underperforming, but in fact after enough iterations,  the sample mean will still fluctuate around $1.4$. In addition, when $|\mu|$ is small and $\sigma$ is large relatively, the performance is  bad. This is because larger $\sigma$ reflects higher level of randomness of the environment, and at this time the significance of exploration becomes smaller.

  The performance under different $\lambda$ with  Gaussian is shown in Figures \ref{lambdalog} and \ref{lambdanolog}.  We can see  that when $\rho^2$ is  relatively larger, $\lambda$ has a more significant impact on algorithm performance under regularizer $\Phi_h$ than $\log\Phi_h$. This is consistent with Remark \ref{costremark}. Finally, we show one sample trajectory of $u_{t_i}$ under different $h$ in Figure \ref{sampledistribution}. It is clearly  from Figure \ref{sampledistribution} that the trajectories of $u_{t_i}$ under different regularizer are different, and the data from exponential distribution is more spread out compared to the normal and uniform distributions. In particular, most data of exponential distribution are small while some data are very large, which may be the reason why exponential distribution sometimes underperforms. Since our parameters and target settings are the same as those in \cite{WZ20}, we can see that   our RL algorithm   based on Choquet regularizations and logarithmic Choquet regularizers perform on par with the one  in \cite{WZ20}.  Compared with the results  that  Gaussian is always the optimal  in  \cite{WZ20}, the availability of a large class of Choquet regularizers makes it possible to choose specific regularizers  to achieve certain objective  used exploratory samplers such as exponential, uniform and Gaussian.

 \begin{figure}[htbp!]
\centering
\begin{minipage}[t]{0.48\textwidth}
\centering
\includegraphics[width=8cm]{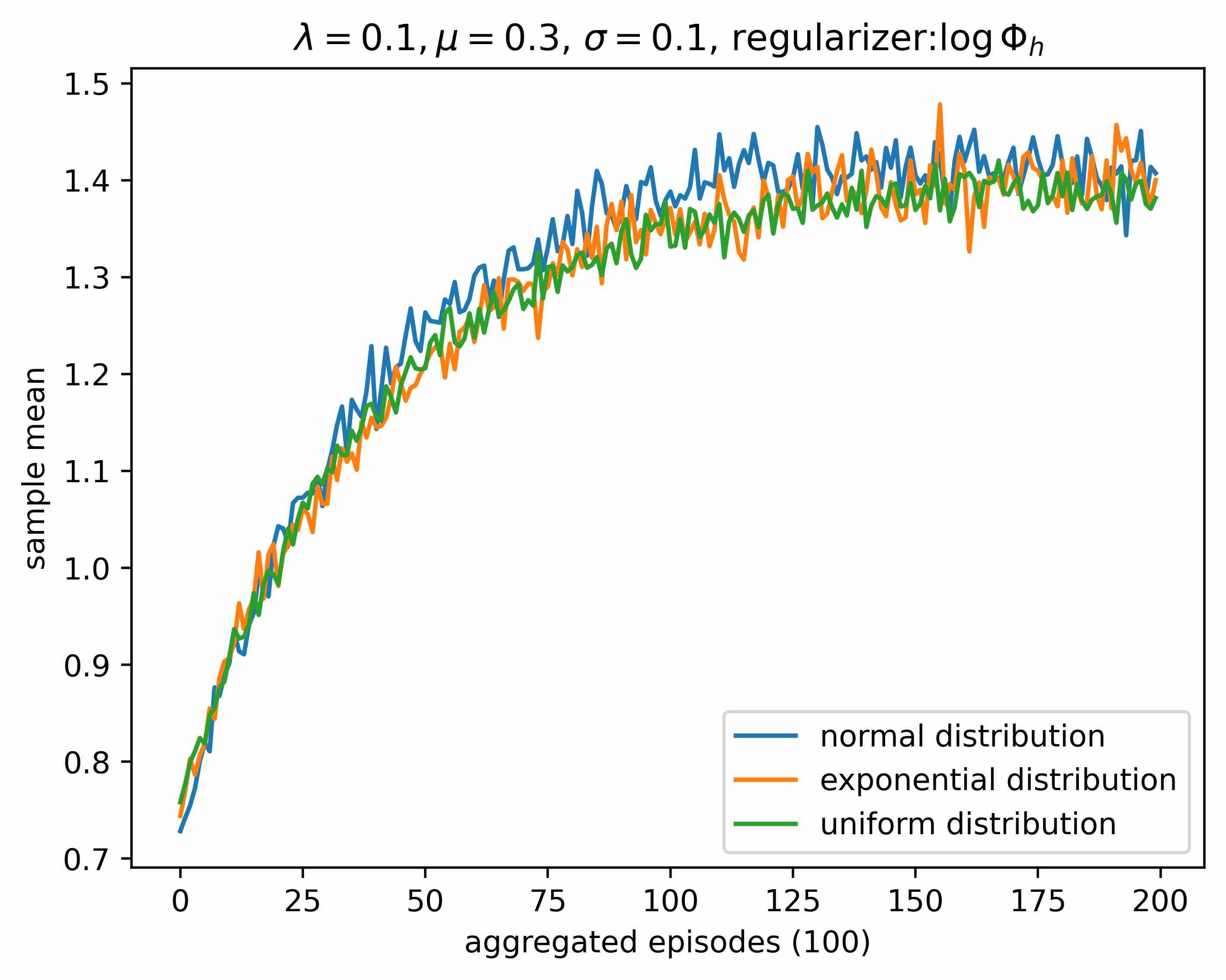}
\end{minipage}
\begin{minipage}[t]{0.48\textwidth}
\centering
\includegraphics[width=8cm]{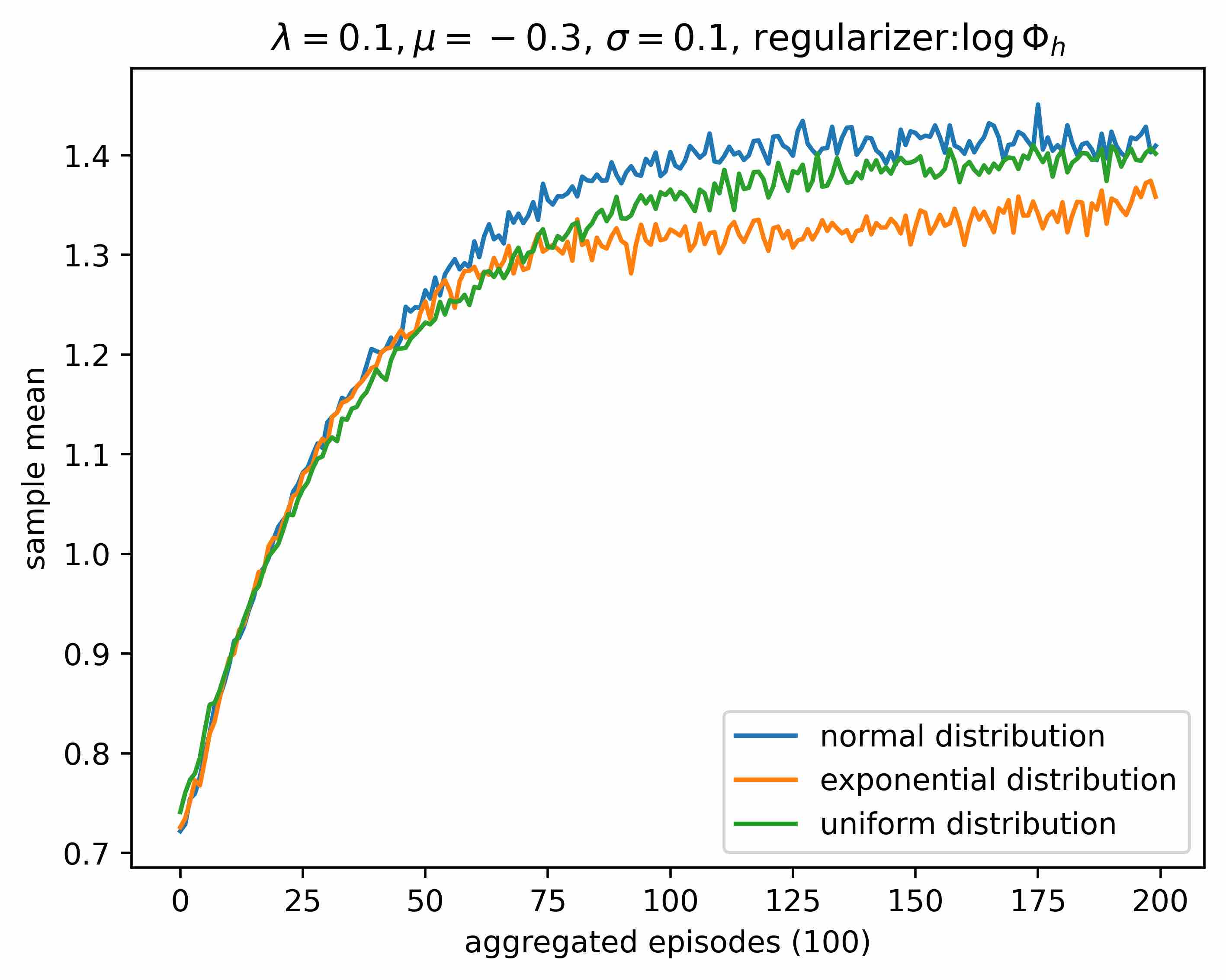}
\end{minipage}
 \quad   
 \begin{minipage}[t]{0.48\textwidth}
\centering
\includegraphics[width=8cm]{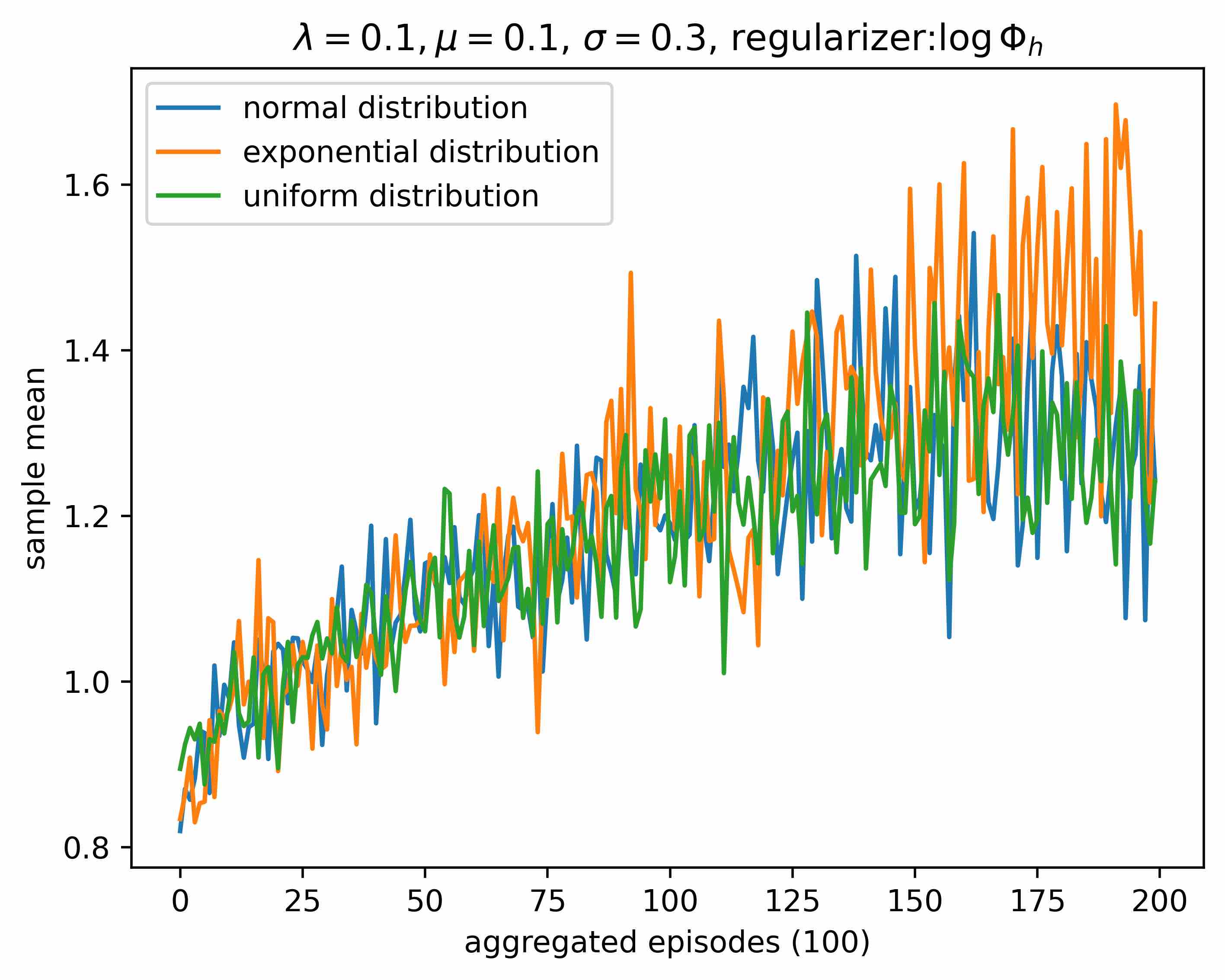}
\end{minipage}
\begin{minipage}[t]{0.48\textwidth}
\centering
\includegraphics[width=8cm]{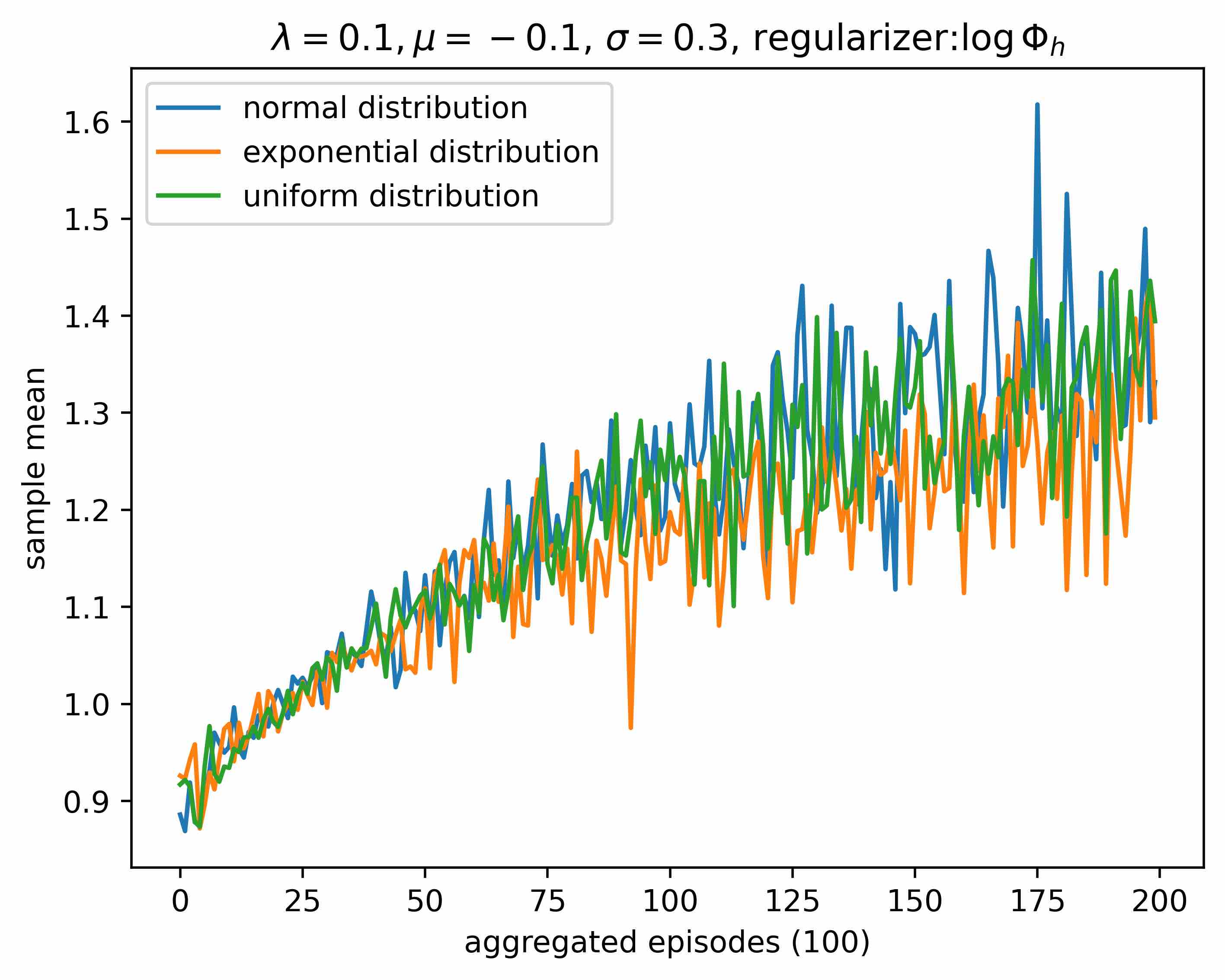}
\end{minipage}
\caption{The effect of $\mu$ and $\sigma$ on the exploration for the regularizer $\log\Phi_h$}\label{logtend}

\end{figure}
 
 \begin{figure}[htbp!]
\centering
\begin{minipage}[t]{0.48\textwidth}
\centering
\includegraphics[width=8cm]{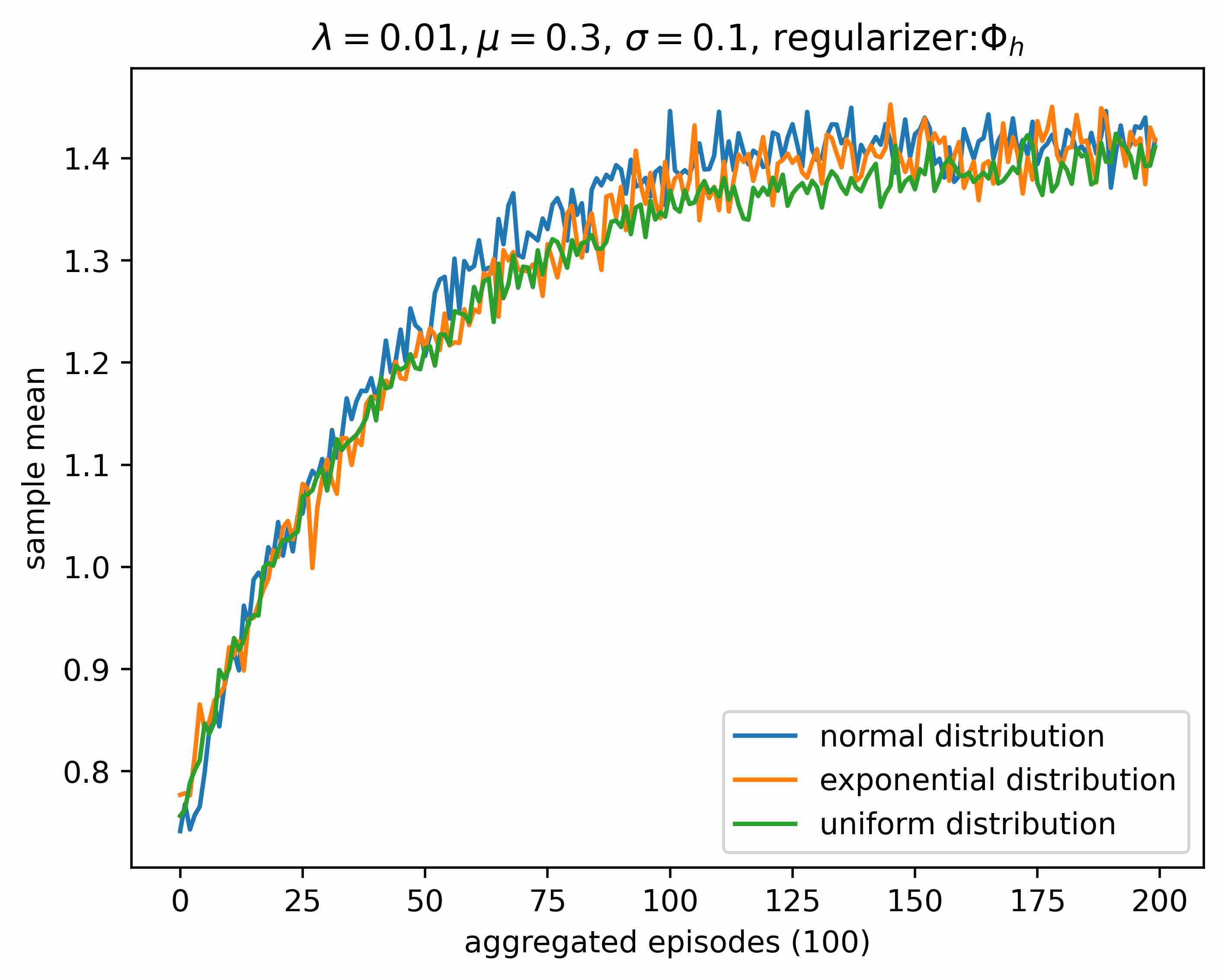}
\end{minipage}
\begin{minipage}[t]{0.48\textwidth}
\centering
\includegraphics[width=8cm]{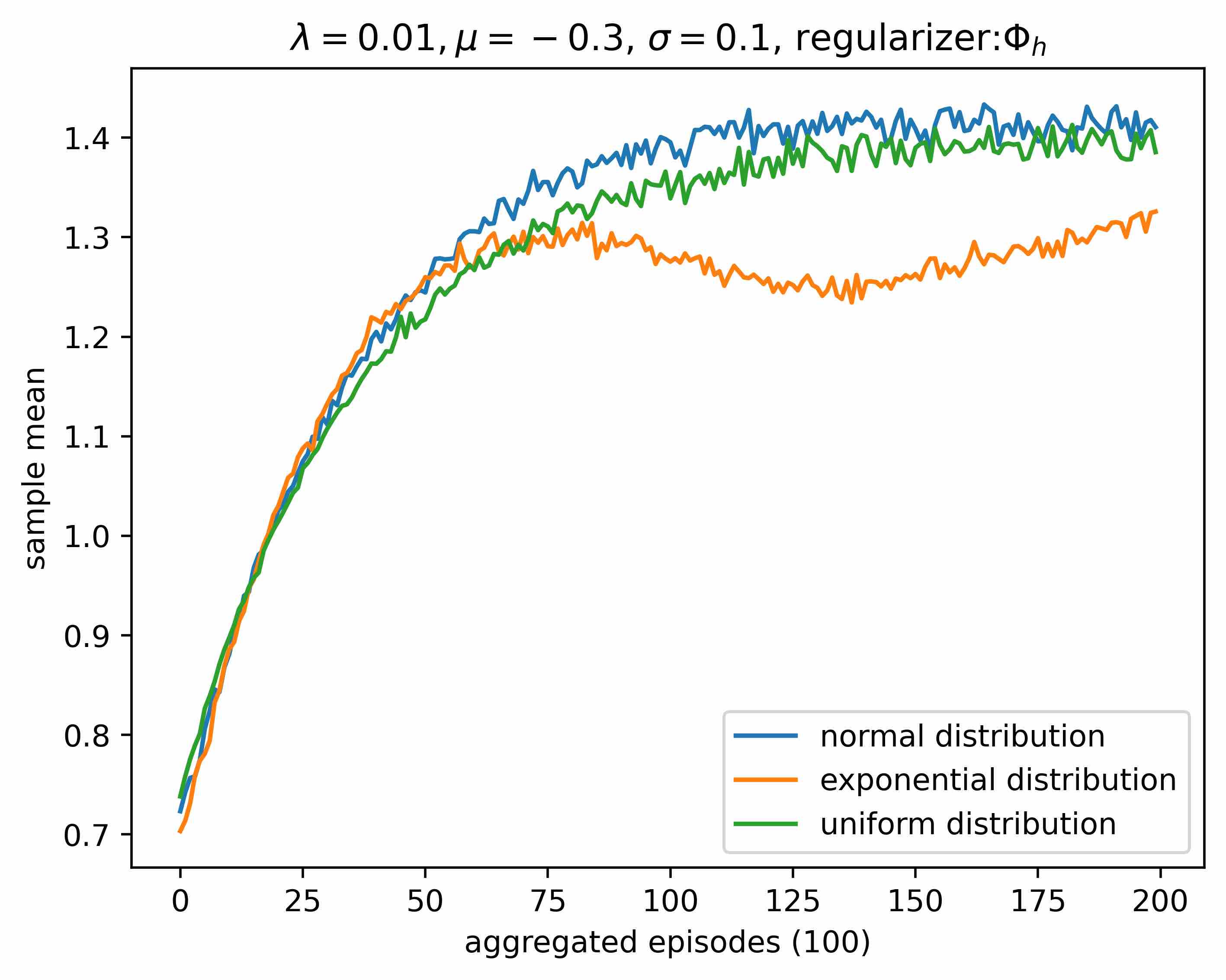}
\end{minipage}
 \quad   
 \begin{minipage}[t]{0.48\textwidth}
\centering
\includegraphics[width=8cm]{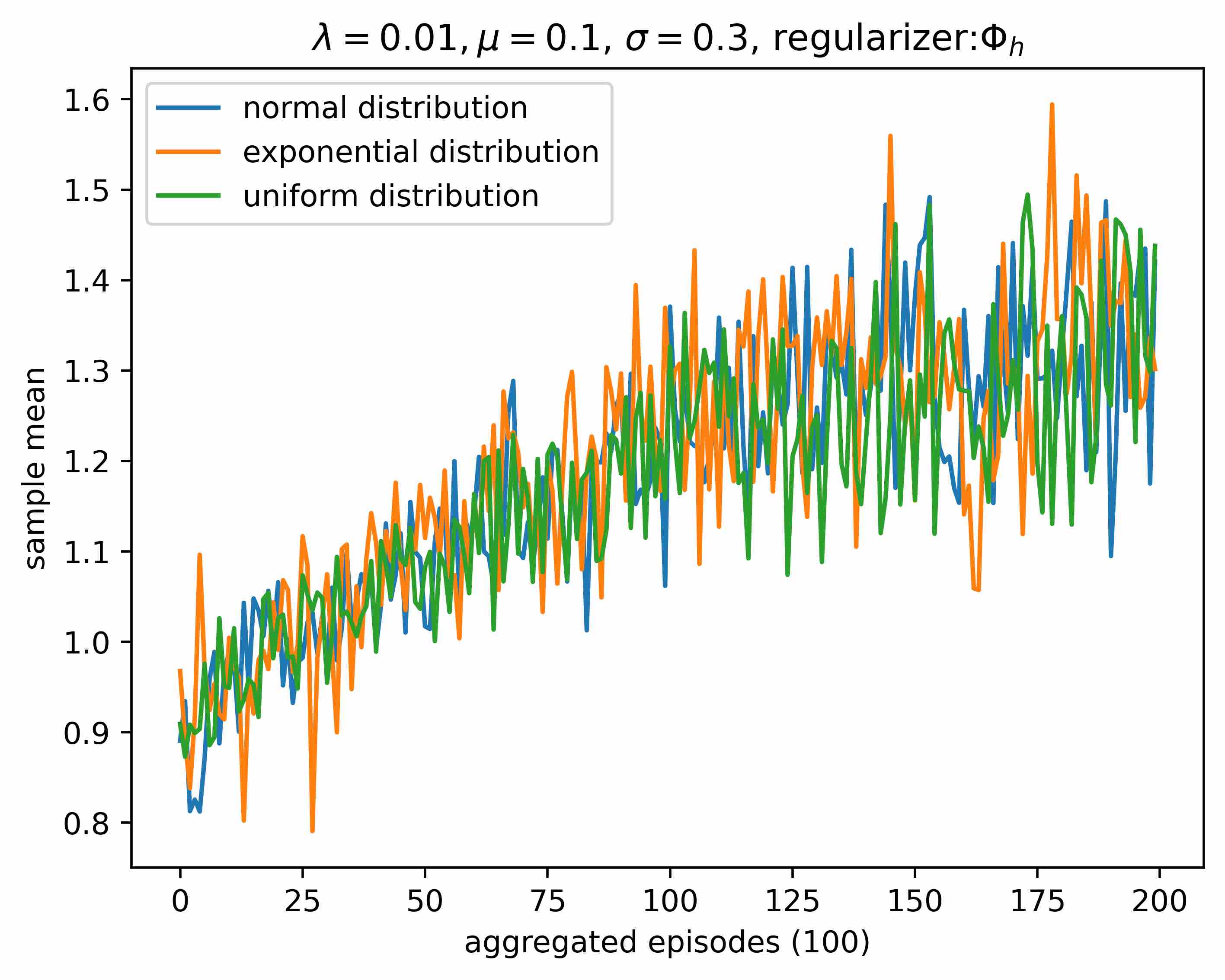}
\end{minipage}
\begin{minipage}[t]{0.48\textwidth}
\centering
\includegraphics[width=8cm]{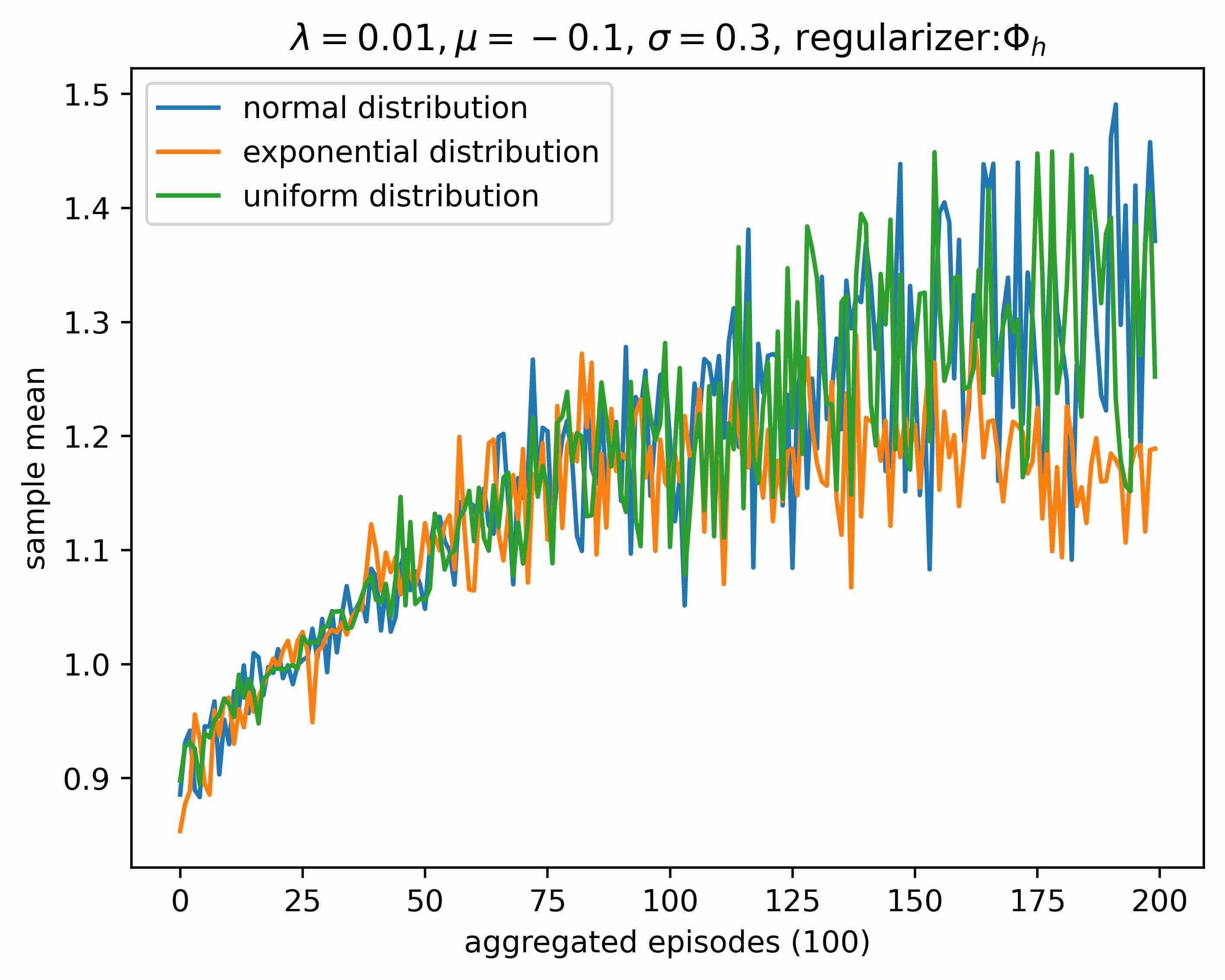}
\end{minipage}
\caption{The effect of $\mu$ and $\sigma$ on the exploration for the regularizer $\Phi_h$}
\label{nologtend}
\end{figure}

 \begin{figure}[htbp!]
\centering
\begin{minipage}[t]{0.48\textwidth}
\centering
\includegraphics[width=8cm]{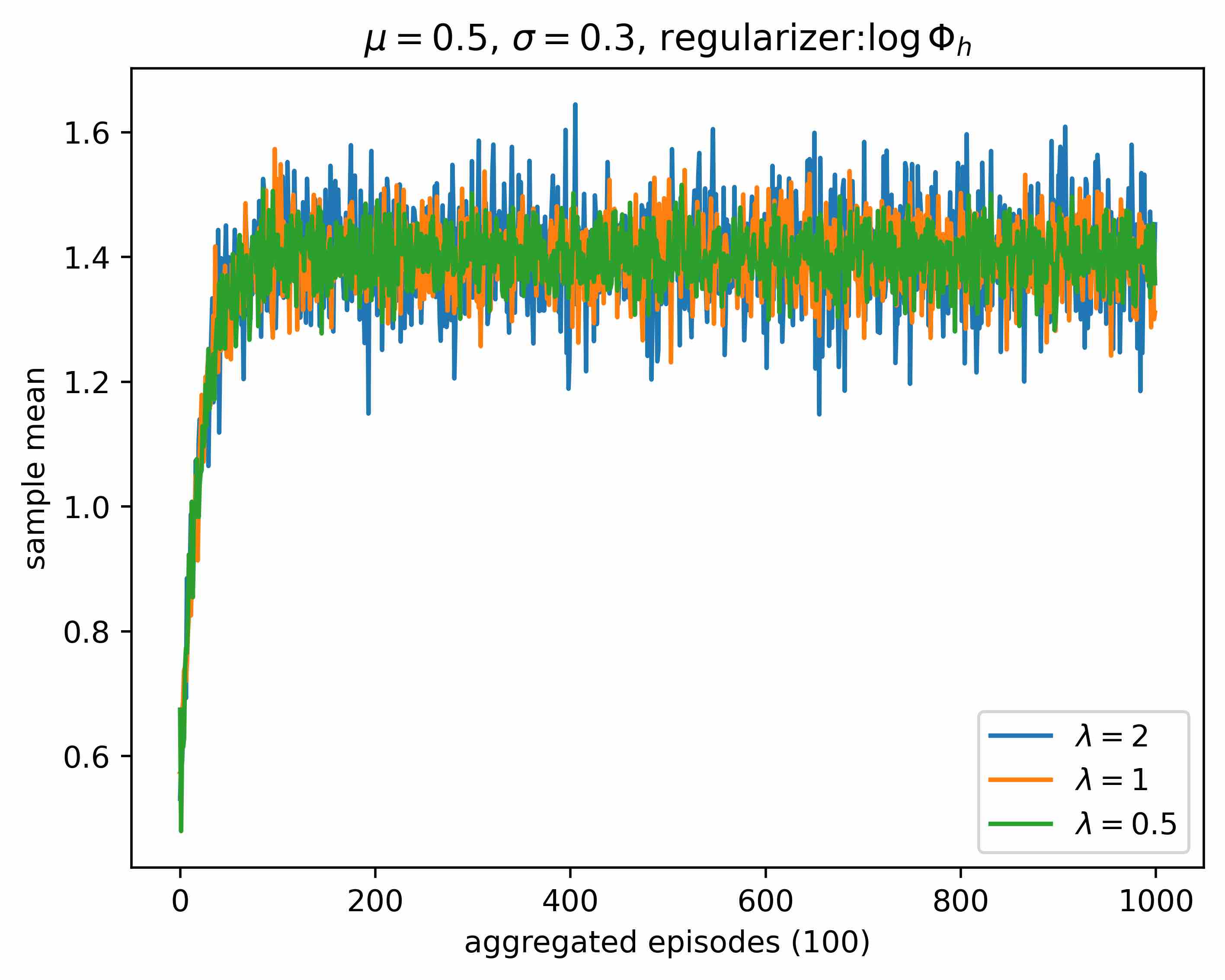}
\end{minipage}
\begin{minipage}[t]{0.48\textwidth}
\centering
\includegraphics[width=8cm]{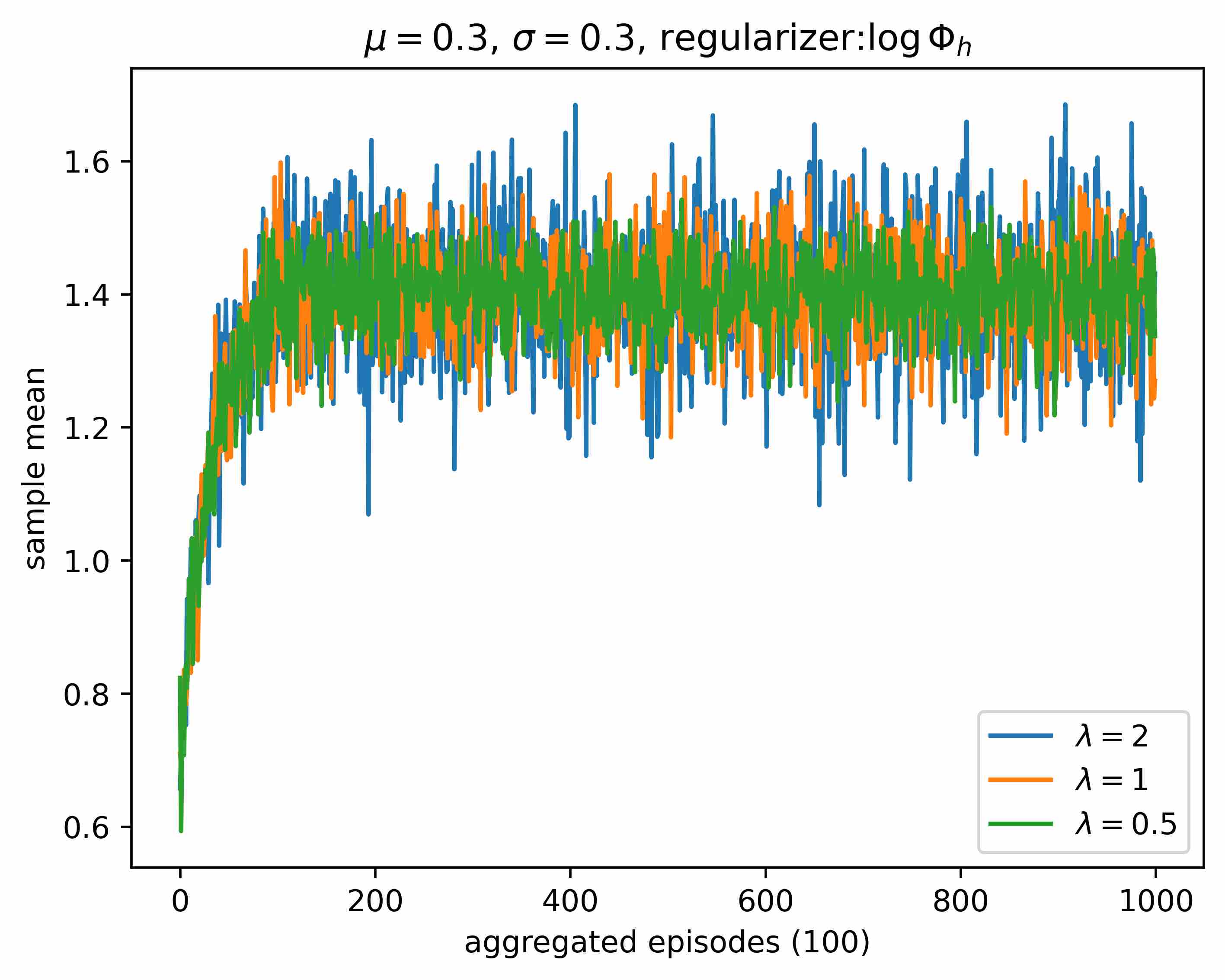}
\end{minipage}
 \quad   
 \begin{minipage}[t]{0.48\textwidth}
\centering
\includegraphics[width=8cm]{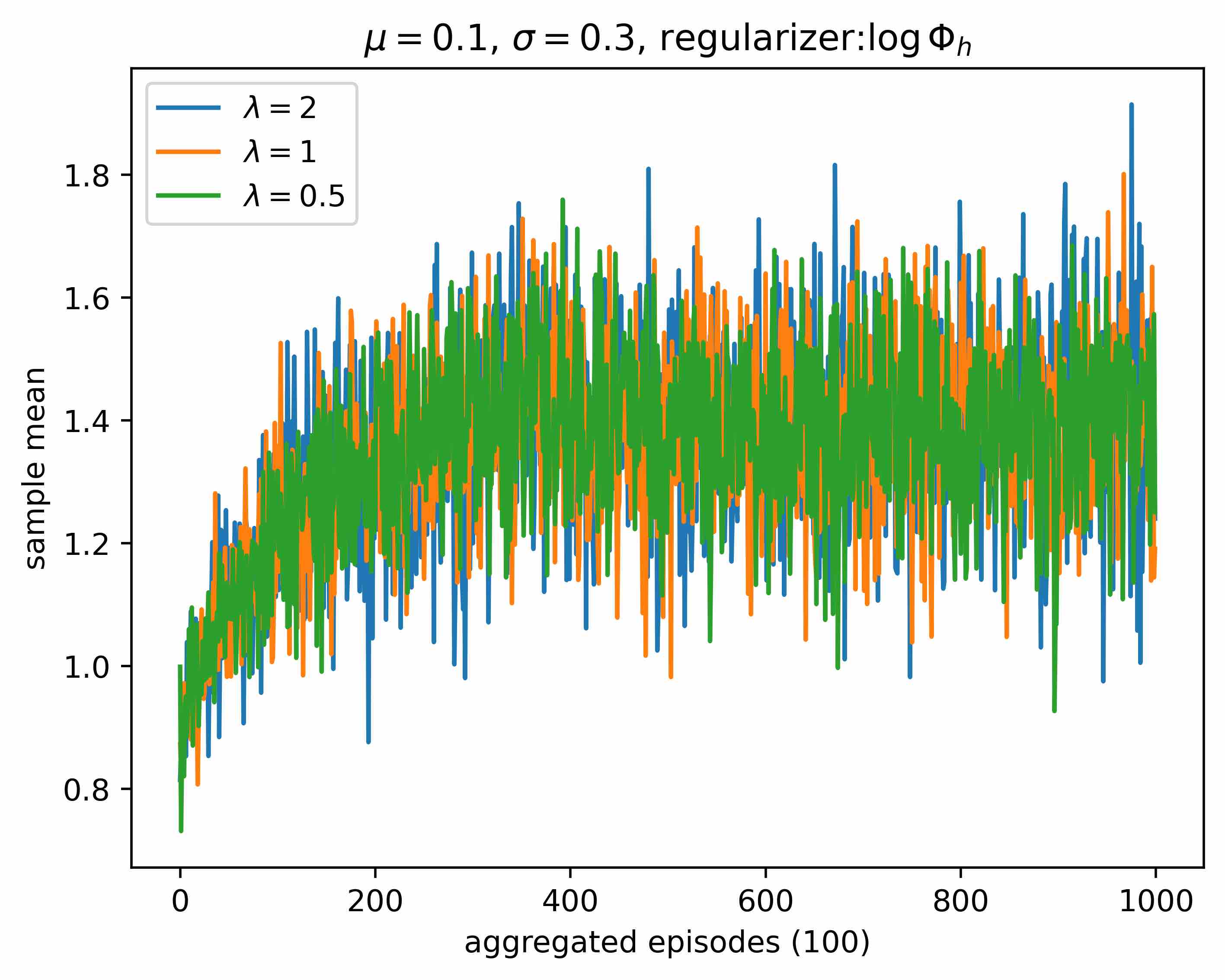}
\end{minipage}
\caption{The effect of  $\lambda$ on the exploration for the regularizer $\log\Phi_h$}
\label{lambdalog}
\end{figure}

\begin{figure}[htbp!]
\centering
\begin{minipage}[t]{0.48\textwidth}
\centering
\includegraphics[width=8cm]{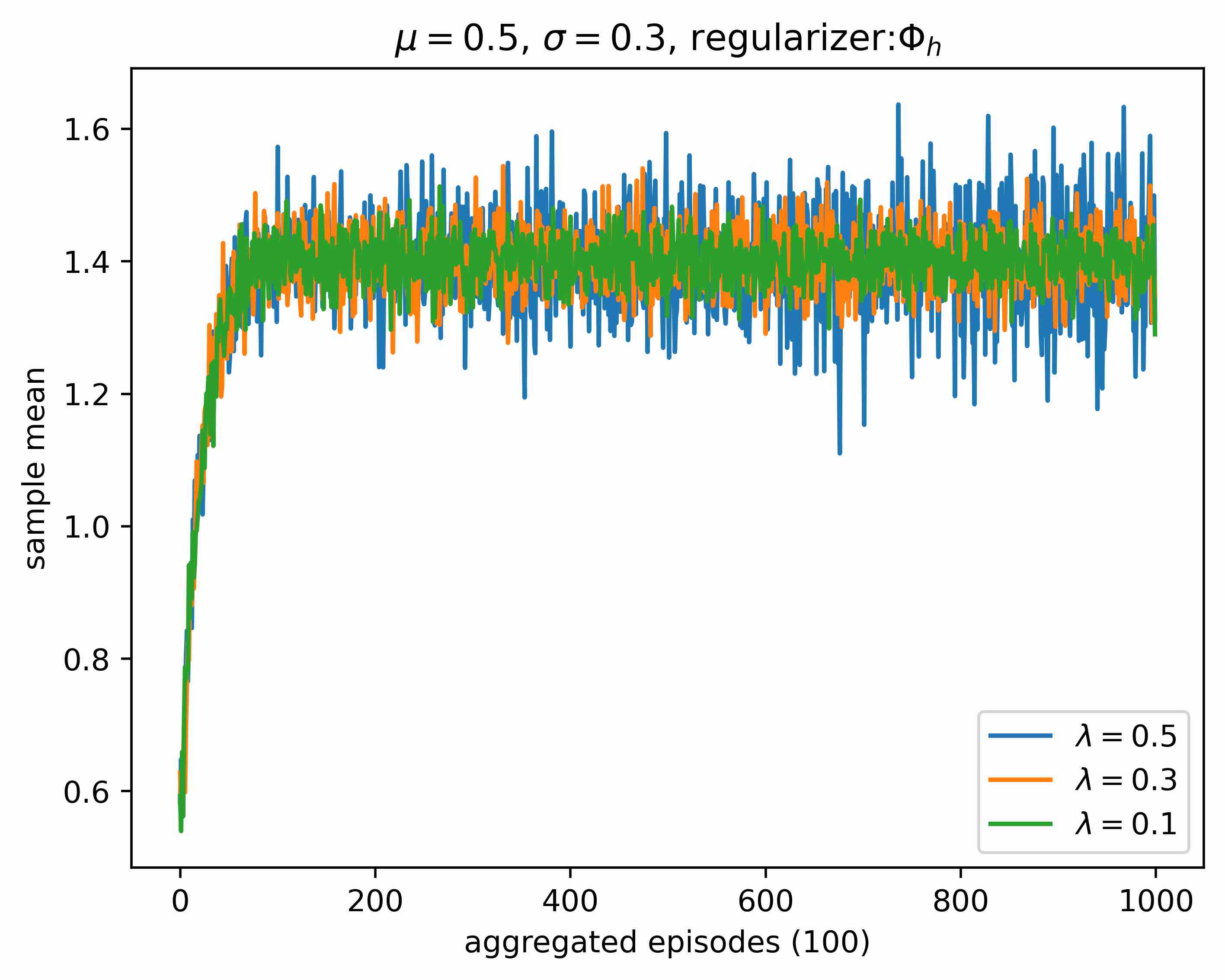}
\end{minipage}
\begin{minipage}[t]{0.48\textwidth}
\centering
\includegraphics[width=8cm]{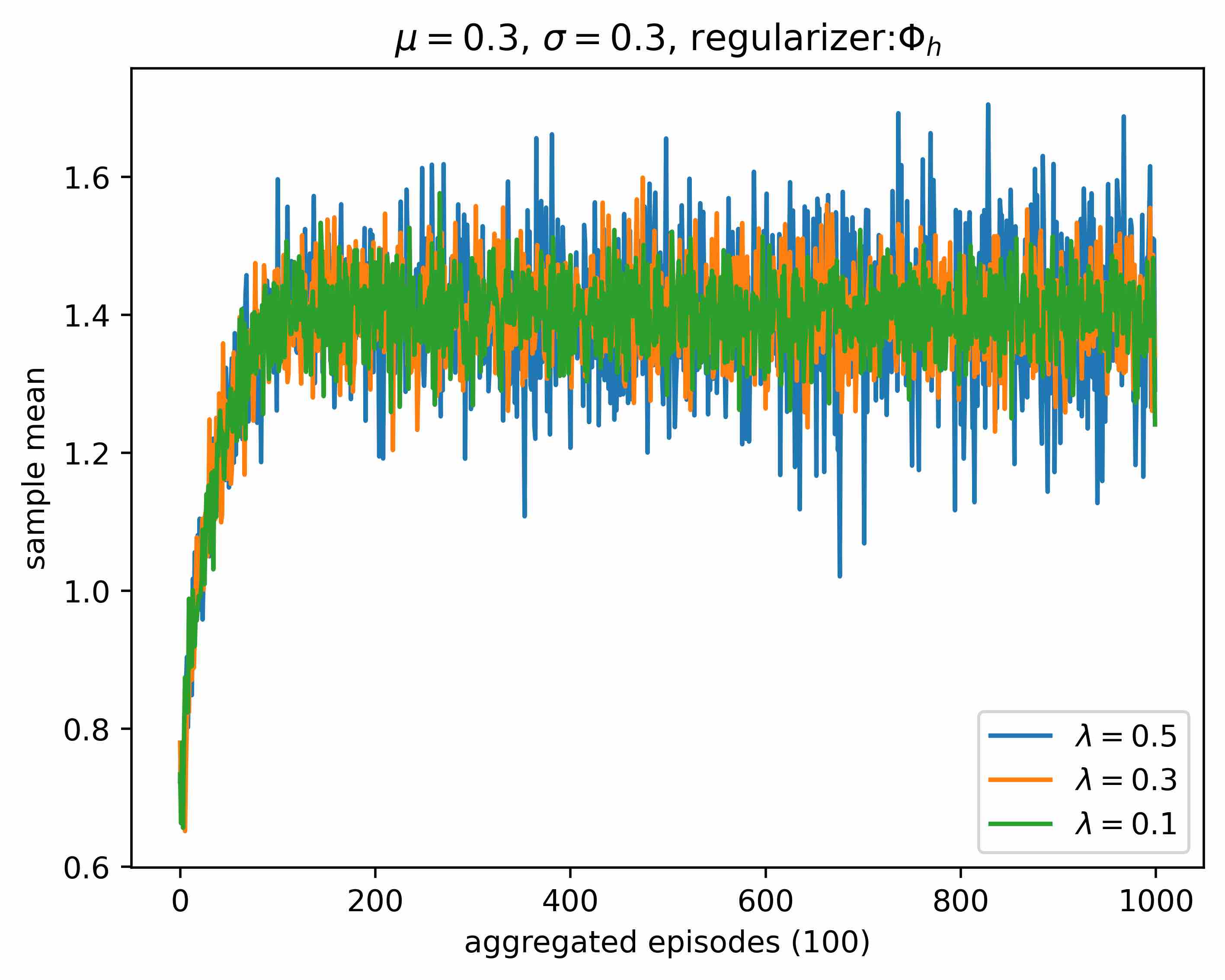}
\end{minipage}
 \quad   
 \begin{minipage}[t]{0.48\textwidth}
\centering
\includegraphics[width=8cm]{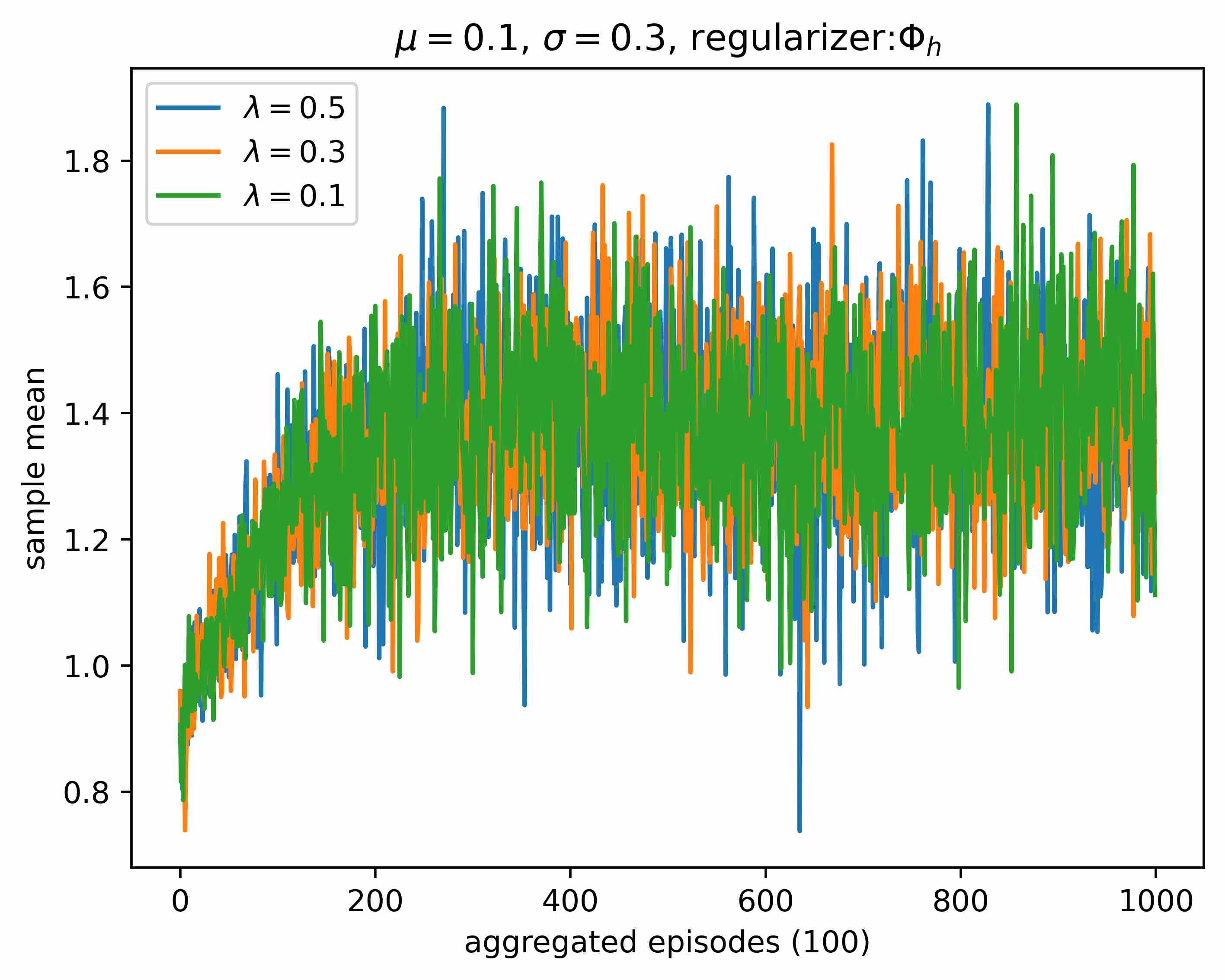}
\end{minipage}
\caption{The effect of  $\lambda$ on the exploration for the regularizer $\Phi_h$}
\label{lambdanolog}
\end{figure}

\begin{figure}[htbp!]
\centering
\begin{minipage}[t]{0.48\textwidth}
\centering
\includegraphics[width=8cm]{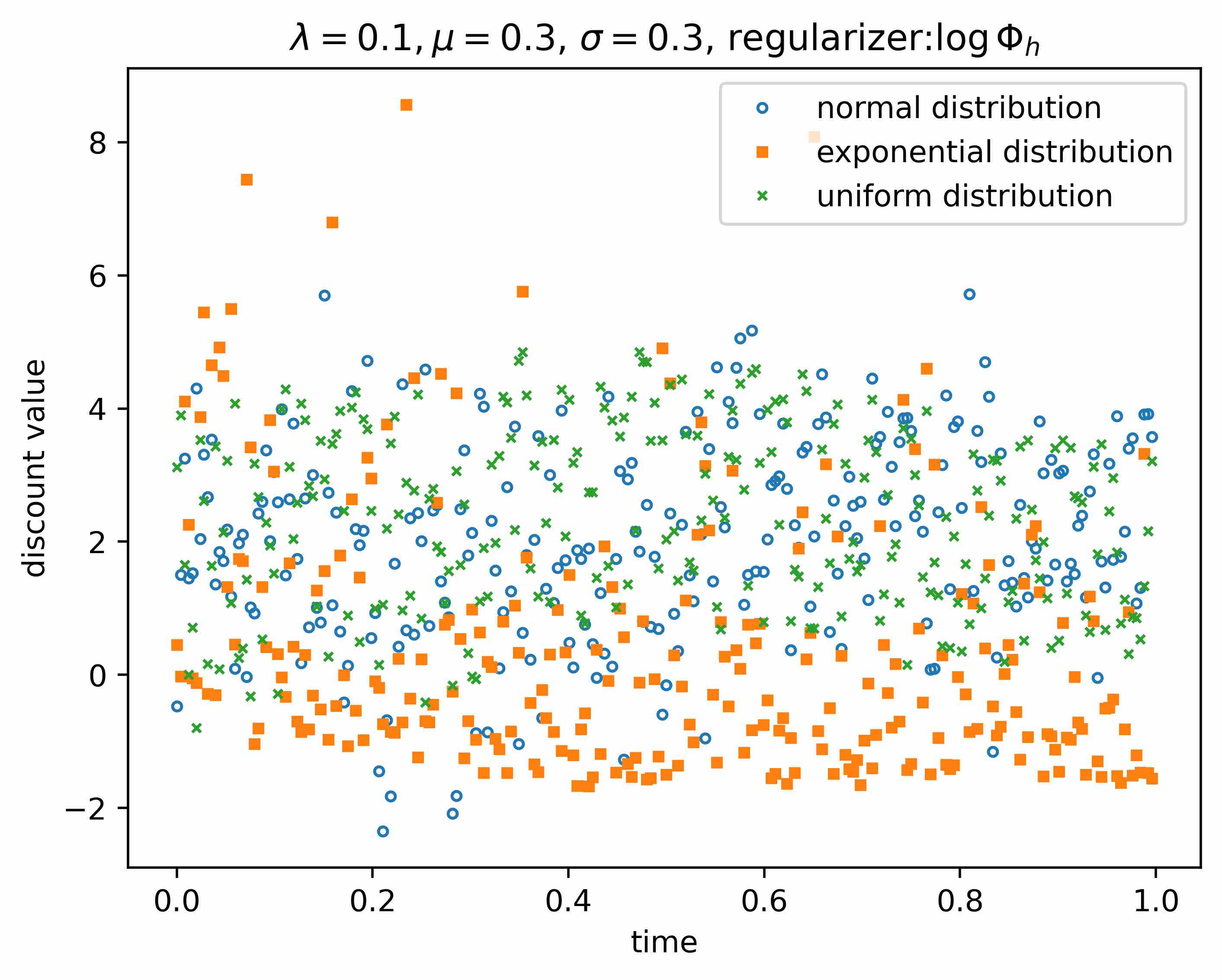}
\end{minipage}
\begin{minipage}[t]{0.48\textwidth}
\centering
\includegraphics[width=8cm]{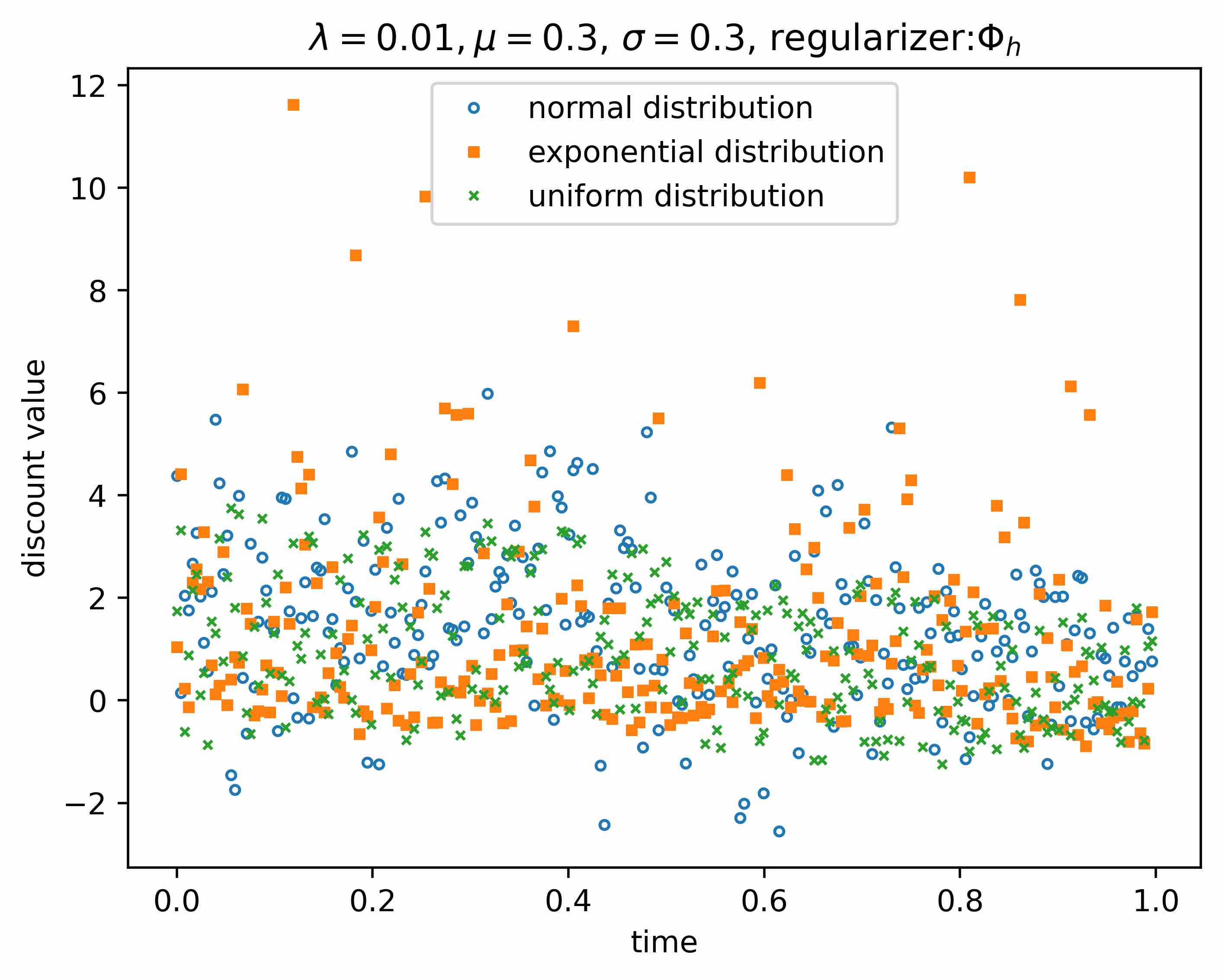}
\end{minipage} 
\caption{Samples of $u_{t_i}$ for the regularizer $\log\Phi_h$  and $\Phi_h$}
\label{sampledistribution}
\end{figure}

\section{CONCLUSION}\label{sec:7}

For the first time, we applied the Choquet-regularized continuous-time RL framework proposed by \cite{HWZ23} to practical problems. We studied the MV problem under Choquet regularization and its logarithmic form. Several different optimal exploration distributions of different $h$ were given, and when  $\Vert h'\Vert_2$ is fixed, the optimal exploration distributions have the same mean and variance. Unlike the infinite time horizon results in \cite{HWZ23},  the variance decreases over time in the finite time horizon problem. At the same time, the mean of the optimal exploration distribution is related to the current state $x$ and independent of $\lambda$ and $h$, which is equal to the optimal action of the classical MV problem. The variance of the optimal exploration distribution is related to $\lambda$ and $h$ and independent of state $x$, and even independent of $h$ under logarithmic regularization. These also showed the perfect separation between exploitation and exploration in the mean and variance of the optimal  distributions as in \cite{WZ20} when entropy is used as a regularizer. 

Further, we have obtained that the two regularization problems converge to the traditional MV problem, and compared the exploration costs of the two regularizations. We found that the exploration cost under the logarithmic Choquet regularization is consistent with the exploration cost under the entropy regularization,  only related to $\lambda$ and time range $T$, while the exploration cost under Choquet regularization is also related to market parameters. Through simulation, we compared the two kinds of regularization. In general, when the market fluctuates greatly and the willingness to explore is not strong, the cost of Choquet regularization is lower. On the contrary, it may be better to use logarithmic Choquet regularizers  for regularization. 

There are still some open questions. First of all, we regard $\lambda$ as an exogenous variable. From the perspective of exploration cost, turning $\lambda$ into endogenous and changeable can help us better control the exploration cost. As time goes by, the information we obtain through exploration will also increase, so the willingness to explore will also change, which also implies the rationality of the changing $\lambda$ to time-related. Secondly, the current Choquet integral can only deal with one-dimensional action space, thus how to extend the Choquet regularizers  to multi-dimensional situations to adapt to more problems is still a challenging problem. We will study these issues in the future.

\vspace{0.7cm}
\noindent\textbf{Acknowledgements.}~~This work was supported by the National Natural Science Foundation of China (No. 11931018 and 12271274)

\appendix

\end{document}